\documentclass[a4paper,12pt]{article}
\usepackage[latin1]{inputenc}
\usepackage[T1]{fontenc}
\usepackage[francais]{babel}
\usepackage{amsmath}
\usepackage{amsfonts}
\usepackage{amssymb}
\usepackage{amsthm}
\usepackage{mathrsfs}
\usepackage{graphicx}
\usepackage[all]{xy}
\theoremstyle{definition}
\newtheorem{definition}{\textbf{Définition}}[section]

\newtheorem{remarque}[definition]{\textbf{Remarque}}

\theoremstyle{plain}

\newtheorem{theoreme}[definition]{\textbf{Théorème}}
\newtheorem{prop}[definition]{\textbf{Proposition}}
\newtheorem{corollaire}[definition]{\textbf{Corollaire}}
\newtheorem{lemme}[definition]{\textit{Lemme}}

\DeclareMathOperator{\e}{e}
\DeclareMathOperator{\li}{li}
\newcommand{\bfx}{\boldsymbol{x}}
\newcommand{\bfy}{\boldsymbol{y}}
\newcommand{\bfu}{\boldsymbol{u}}

\title{\Huge{\bf{Sur le nombre d'idéaux dont la norme  est la valeur d'une forme binaire de degré $3$}}}
\author{Alexandre Lartaux}

\begin{document}

\maketitle

\vglue0.3cm
\hglue0.02\linewidth\begin{minipage}{0.9\linewidth}
\begin{center}
{Universit\'e de Paris, Sorbonne Université}\\
CNRS,  \\
 Institut de~Math\'ematiques~de Jussieu- Paris Rive Gauche,\\
F-75013 Paris, France \\
 E-mail : \parbox[t]{0.45\linewidth}{\texttt{alexandre.lartaux@imj-prg.fr}}

\end{center}
\end{minipage}

\renewcommand{\abstractname}{Abstract}
\begin{abstract}
Let $\mathbb{K}$ be a cyclic extension of degree $3$ of $\mathbb{Q}$. Take $G={\rm Gal}(\mathbb{K}/ \mathbb{Q})$ and $\chi$ the character of a non trivial representation of $G$. In this case, $\chi$ is a non principal Dirichlet character of degree $3$ and the quantity $r_3(n)$ defined by
$$r_3(n):=\big(1*\chi*\chi^2\big)(n){\rm ,}$$
counts the number of ideals of $O_{\mathbb{K}}$ of norm $n$. In this paper, using a new result on Hooley's Delta function from \cite{L}, we prove an asymptotic estimate, in $\xi$, of the quantity
$$Q(\xi,\mathcal{R},F):=\sum\limits_{\bfx \in \mathcal{R}(\xi)}{r_3\big(F(\bfx)\big)}{\rm ,}$$
for a binary form $F$ of degree $3$ irreducible over $\mathbb{K}$ and $\mathcal{R}$ a good domain of $\mathbb{R}^2$, with
$$\mathcal{R}(\xi):=\Big\{\bfx \in \mathbb{R}^2\;:\: \frac{\bfx}{\xi} \in \mathcal{R}\Big\}{\rm .}$$
We also give a geometric interpretation of the main constant of the asymptotic estimate when the ring $O_{\mathbb{K}}$ is principal.
\end{abstract}

\tableofcontents

\section{Introduction et résultats}
Lorsque $\mathcal{R}$ désigne un domaine de $\mathbb{R}^2$ et $\xi \in \mathbb{R}_+^*$, nous désignons par $\mathcal{R}(\xi)$ le domaine de $\mathbb{R}^2$ obtenu en dilatant $\mathcal{R}$ par $\xi$, c'est-à-dire
$$\mathcal{R}(\xi):=\{\bfx \in \mathbb{R}^2 : \bfx/\xi \in \mathcal{R}\}{\rm .}$$
 Soit $\mathbb{K}$ une extension cyclique de $\mathbb{Q}$ de degré $3$, soit $G=\rm{Gal}(\mathbb{K}/\mathbb{Q})$ son groupe de Galois et $\chi$  un caractère non trivial de $G$. L'application $\chi$ est donc un caractère de Dirichlet non trivial d'ordre $3$ et nous notons par la suite $q$ son conducteur. \\
Rappelons que pour $f$ et $g$ deux fonctions arithmétique, le produit de convolution $f*g$ est la fonction arithmétique définie par
$$(f*g)(n)=\sum\limits_{d\mid n}{f(d)g\Big(\frac{n}{d}\Big)} \;\;\;\;\; \forall n\geqslant 1{\rm .}$$
Nous définissons pour la suite la fonction arithmétique $r_3$ par 
\begin{equation}
\label{r}
r_3(n):=(1*\chi*\chi^2)(n){\rm ,}
\end{equation}
où la dépendance en $\chi$ est omise. La quantité $r_3(n)$ compte le nombre d'idéaux de l'anneau des entiers du corps de nombre $\mathbb{K}$ dont la norme vaut $n$. Cela provient du théorème $6$ de \cite{H}.\\
 
Nous désignons par $F \in \mathbb{Z}[X,Y]$ une forme binaire, homogène de degré~$3$. Nous proposons dans cet article une estimation asymptotique en $\xi$ de la quantité $Q(F,\mathcal{R},\xi)$ définie de la manière suivante

\begin{equation}
\label{eq Q}
Q(F,\xi,\mathcal{R}):=\sum\limits_{\bfx \in \mathbb{Z}^2\cap \mathcal{R}(\xi)}{r_3(F(\bfx))}
\end{equation}
lorsque le domaine $\mathcal{R}$ et la forme binaire $F$ vérifient les hypothèses suivantes :
\begin{align*}
\text{(H1)} & &&\text{Le domaine } \mathcal{R} \text{ est un ouvert borné convexe dont } \\
& &&\text{la frontière est continuement différentiable{\rm ;}} \\
\text{(H2)} & &&\forall \bfx \in \mathcal{R}\mbox{, } ||\bfx||\leqslant \sigma {\rm ;} \\
\text{(H3)} & && \forall \bfx \in \mathcal{R}\mbox{, } |F(\bfx)|\leqslant \vartheta^3 {\rm ;}\\
\text{(H4)} & && \text{La forme } F \text{ est irréductible sur } \mathbb{K} {\rm ,}
\end{align*}
pour certaines valeurs $\sigma>0$ et $\vartheta>0$. Cela peut être vu comme une version analogue de l'étude de \cite{B-T} qui étudie la somme de $(1*\chi)(T(\bfx))$ lorsque $\chi$ est le caractère non principal modulo $4$ et $T$ une forme binaire quartique.\\
Nous introduisons les ensembles 
\begin{equation}
\label{Epsilon}
\mathcal{E}:=\bigcup\limits_{\alpha \in G_1}\{n \in \mathbb{N}^*\mbox{ : } \exists d\mid q^{\infty}\mbox{, } n\equiv \alpha d \bmod dq\}
\end{equation}
où 
$$G_1:={\rm Ker}(\chi)\subset (\mathbb{Z}/q\mathbb{Z})^{\times}$$ 
et pour $d \mid q^{\infty}$, nous notons 
$\mathcal{E}_d$ la projection de $\mathcal{E}$ sur $\mathbb{Z}/d\mathbb{Z}$, c'est-à-dire
\begin{equation}
\label{Epsilond}
\mathcal{E}_d=\bigcup\limits_{\alpha \in G_1}\{n \in \mathbb{Z}/d\mathbb{Z}\mbox{ : } \exists d_1\mid q^{\infty}\mbox{, } n\equiv \alpha d \bmod (d_1q,d)\}{\rm .}
\end{equation}
Enfin, nous définissons l'exposant
\begin{equation}
\label{eta}
\eta:=0,0034{\rm .}
\end{equation}
Rappelons que pour $\{s\;:\;\mathfrak{Re}(s)>0\}$ et $\chi$ un caractère de Dirichlet non principal, la fonction $L$ de Dirichlet est définie par
$$L(s,\chi)=\sum\limits_{n\geqslant 1}{\frac{\chi(n)}{n^s}}{\rm .}$$
\begin{theoreme}
\label{theo 1}
Soient  $\xi>0$, $\mathbb{K}$ une extension cyclique de $\mathbb{Q}$ de degré $3$, $\chi$ un caractère non principal de $G$, $F \in \mathbb{Z}[X,Y]$ une forme binaire de degré $3$,\\
\noindent
 $\mathcal{R}$ un domaine de $\mathbb{R}^2$. Pour tout $\varepsilon>0$, $\vartheta>0$ et $\sigma>0$ tels que les hypothèses (H1), (H2), (H3) et (H4) sont vérifiées et sous les conditions
$$1/\sqrt{\xi}\leqslant \sigma \leqslant \xi^{3/2}{\rm ,}\;\;\;\;\; 1/\sqrt{\xi}\leqslant \vartheta \leqslant \xi^{3/2}{\rm ,}$$
nous avons
\begin{equation}
\label{1}
Q(F,\xi,\mathcal{R})=K(F)L(1,\chi)L(1,\chi^2){\rm vol}(\mathcal{R})\xi^2+O\Big(\frac{||F||^{\varepsilon}(\sigma^2+\vartheta^2)\xi^2}{(\log \xi)^{\eta}}\Big)
\end{equation}
où $||F||$ désigne le maximum des coefficients de $F$ et
\begin{equation}
K(F):=K_q(F)\prod\limits_{p\nmid q}{K_p(F)}
\end{equation}
avec 
\begin{equation}
\label{K_p}
K_p(F):=\Big(1-\frac{\chi(p)}{p}\Big)\Big(1-\frac{\chi^2(p)}{p}\Big)\sum\limits_{\nu\geqslant 0}{\frac{\varrho_F^+(p^{\nu})}{p^{2\nu}}(\chi*\chi^2)(p^{\nu})}
\end{equation}
pour tout $p$ premier avec $q$ et 
\begin{equation}
\label{K_q}
K_q(F)=\lim\limits_{k\rightarrow \infty}\frac{3}{q^{2k}}\Big|\Big\{\bfx \in (\mathbb{Z}/q^k\mathbb{Z})^2\mbox{ : } F(\bfx) \in \mathcal{E}_{q^k}\Big\}\Big|{\rm ,}
\end{equation}
où $\mathcal{E}_{q^k}$ est défini en \eqref{Epsilond}.
\end{theoreme}
\begin{remarque}
Si de plus l'anneau des entiers $O_{\mathbb{K}}$ est principal, nous pouvons décomposer $K(F)$ en produit de facteurs non archimédiens qui possèdent chacun une interprétation géométrique. Nous étudierons ce cas dans la section $5$.
\end{remarque}
Ce résultat repose sur la méthode utilisée dans \cite{B} et nécessite le Théorème~\ref{theo L}. Cependant, plusieurs difficultés techniques apparaissent, d'une part, le conducteur $q$ de $\chi$ ne possède pas forcément un unique facteur premier, ce qui complique la paramétrisation de la section 3.2, d'autre part, l'exposant $\rho$ apparaissant dans le Théorème \ref{theo L} n'est pas suffisamment précis pour adapter directement les calculs menés dans \cite{B}. \\

Nous trouverons dans l'appendice B.$4$ de \cite{C} une liste de corps $\mathbb{K}$ vérifiant les hypothèses du Théorème \ref{theo 1}. On peut ainsi choisir $\mathbb{K}=\mathbb{Q}[\alpha]$ où $\alpha$ est une racine d'un des polynômes suivants : $X^3+X^2-2X-1$, $X^3-3X-1$, $X^3-X^2-4X-1$.

\section{Rappels}
\subsection{\'{E}tude du nombre de solutions locales d'une équation polynomiale}
Nous rappelons dans cette section les résultats préliminaires et les notations nécessaires à la démonstration du Théorème \ref{theo 1}. Lorsque $F \in \mathbb{Z}[X,Y]$ est une forme binaire homogène de degré $3$, nous notons $\mathcal{D}(F):={\rm disc}(F)$ son discriminant 
$$\mathcal{D}(F(X,Y))=\mathcal{D}(F(X,1))=\mathcal{D}(F(1,Y)) {\rm .}$$
Lorsque $s\in \mathbb{N}^*$, nous posons
$$\varrho_F^-(s):=\sum\limits_{\substack{1\leqslant a\leqslant s \\ F(a,1)\equiv 0\bmod s}}{1}{\rm ,}\;\;\;\;\;\;\;\;\;\;\;\;\;\;\;\;\;\;\;\; \varrho_F^+(s):=\sum\limits_{\substack{1\leqslant a,b\leqslant s \\ F(a,b)\equiv 0\bmod s}}{1}{\rm .}$$
Nous introduisons ensuite les ensembles
\begin{equation}
\label{lambda}
\Lambda(s,F):=\{(m,n)\in \mathbb{Z}^2 \,:\, s\mid F(m,n)\}
\end{equation}
et
\begin{equation}
\label{lambda*}
\Lambda^*(s,F):=\{(m,n)\in \Lambda(s,F) \,:\, (m,n,s)=1\}
\end{equation}
et nous posons
$$\varrho_F^*(s):=|\Lambda^*(s,F)\cap[0,s[^2|{\rm .}$$
Lorsque le polynôme $F$ est irréductible sur $\mathbb{Q}$, nous pouvons considérer $k$ un corps de rupture du polynôme $F(X,1)$ et la fonction zêta de Dedekind, $\zeta_k$ définie sur le demi-plan $\{s \in \mathbb{C}\mbox{ : }\Re e(s)>1\}$. Les résultats suivants sont établis dans \cite{B}.
\begin{prop}
Soient $F \in \mathbb{Z}[X,Y]$ une forme binaire de degré $3$ et $\mathbb{K}$ une extension cyclique de $\mathbb{Q}$ de degré $3$ tels que $F$ soit irréductible sur $\mathbb{K}$. Il existe des fonctions multiplicatives $h_F^-$, $h_F^+$ et $h_F^*$ telles que pour tout $\Re e(s)>1$, nous avons
\begin{align}
\sum\limits_{n\geqslant 1}{\frac{\varrho_F^-(n)}{n^s}}=\zeta_k(s)\sum\limits_{n\geqslant 1}{\frac{h_F^-(n)}{n^s}} \\
\sum\limits_{n\geqslant 1}{\frac{\varrho_F^+(n)}{n^{s+1}}}=\zeta_k(s)\sum\limits_{n\geqslant 1}{\frac{h_F^+(n)}{n^s}} \\
\sum\limits_{n\geqslant 1}{\frac{\varrho_F^*(n)}{n^{s+1}}}=\zeta_k(s)\sum\limits_{n\geqslant 1}{\frac{h_F^*(n)}{n^s}}{\rm .}
\end{align}
De plus, ces fonctions multiplicatives vérifient, pour tout $\kappa \in ]0,1/6[$ et pour tout $\varepsilon>0$
\begin{equation}
\sum\limits_{n\geqslant 1}{\frac{|h_F^-(n)|+|h_F^+(n)|+|h_F^*(n)|}{n^{1-\kappa}}}\ll ||F||^{\varepsilon}{\rm .}
\end{equation}
\end{prop}
Le comportement moyen de $(\chi*~\chi^2)(n)\varrho_F^{\pm}(n)$ découle des propriétés analytiques des fonctions $L_{k}(s,\chi)$ et $L_{k}(s,\chi^2)$ où nous avons posé
$$L_{k}(s,\chi):=\sum\limits_{\mathcal{U}\in \mathcal{I}(\mathcal{O}_{k})}{\frac{\chi(N_{k/\mathbb{Q}}(\mathcal{U}))}{N_{k/\mathbb{Q}}(\mathcal{U})^s}}{\rm ,}$$
où $\mathcal{I}(\mathcal{O}_k)$ désigne l'ensemble des idéaux de l'anneau des entiers algébriques de $k$. Nous pouvons ainsi énoncer une version analogue de la proposition ci-dessus en prenant en compte les caractères $\chi$ et $\chi^2$.
\begin{prop}
Soient $\mathbb{K}$ une extension cyclique de $\mathbb{Q}$ de degré $3$, $\chi$ un caractère non principal de $Gal(\mathbb{K}/\mathbb{Q})$ et $F \in \mathbb{Z}[X,Y]$ une forme binaire de degré $3$ irréductible sur $\mathbb{K}$. Il existe des fonctions multiplicatives $h_F^-(\;\cdot\mbox{ ; }\chi)$, $h_F^+(\;\cdot \mbox{ ; }\chi)$ et $h_F^*(\;\cdot\mbox{ ; }\chi)$ telles que pour tout $\Re e(s)>1$, nous avons
\begin{align}
\sum\limits_{n\geqslant 1}{\frac{\chi(n)\varrho_F^-(n)}{n^s}}=L_k(s,\chi)\sum\limits_{n\geqslant 1}{\frac{h_F^-(n;\chi)}{n^s}} \\
\sum\limits_{n\geqslant 1}{\frac{\chi(n)\varrho_F^+(n)}{n^{s+1}}}=L_k(s,\chi)\sum\limits_{n\geqslant 1}{\frac{h_F^+(n;\chi)}{n^s}} \\
\sum\limits_{n\geqslant 1}{\frac{\chi(n)\varrho_F^*(n)}{n^{s+1}}}=L_k(s,\chi)\sum\limits_{n\geqslant 1}{\frac{h_F^*(n;\chi)}{n^s}}{\rm .}
\end{align}
De plus, ces fonctions multiplicatives vérifient, pour tout $\kappa \in ]0,1/6[$ et pour tout $\varepsilon>0$
\begin{equation}
\sum\limits_{n\geqslant 1}{\frac{|h_F^-(n;\chi)|+|h_F^+(n;\chi)|+|h_F^*(n;\chi)|}{n^{1-\kappa}}}\ll ||F||^{\varepsilon}{\rm .}
\end{equation}
Les mêmes résultats restent valables si l'on remplace $\chi$ par $\chi^2$.
\end{prop}
Rappelons la définition de la fonction logarithme intégral, notée $\li$. Pour tout $x\geqslant 2$,
$$\li(x):=\int_{2}^{x}{\frac{1}{\log t}{\rm d}t} {\rm .}$$ 
Ces deux propositions nous permettent d'énoncer le lemme suivant, établi dans \cite{H}.
\begin{lemme}
Soient $\mathbb{K}$ une extension cyclique de $\mathbb{Q}$ de degré $3$, $\chi$ un caractère non principal de $G=\rm{Gal}(\mathbb{K}/\mathbb{Q})$ et $F \in \mathbb{Z}[X,Y]$ une forme binaire de degré $3$ irréductible sur $\mathbb{K}$. Il existe une constante $c>0$ telle que, uniformément pour $x\geqslant 2$
\begin{equation}
\label{eq 1}
\sum\limits_{p\leqslant x}{\varrho_F^-(p)}=\li(x)+O(x\e^{-c\sqrt{\log x}})
\end{equation}
\begin{equation}
\label{eq 2}
\sum\limits_{p\leqslant x}{\chi(p)\varrho_F^-(p)}=O(x\e^{-c\sqrt{\log x}})
\end{equation}
\begin{equation}
\label{eq 3}
\sum\limits_{p\leqslant x}{\chi^2(p)\varrho_F^-(p)}=O(x\e^{-c\sqrt{\log x}}){\rm .}
\end{equation}
Nous avons les mêmes estimations en remplaçant la fonction $\varrho_F^-(p)$ par $\frac{\varrho_F^+(p)}{p}$ ou $\frac{\varrho_F^*(p)}{p}$.
\end{lemme}

\begin{lemme}
Soient $\varepsilon>0$, $\mathbb{K}$ une extension cyclique de $\mathbb{Q}$ de degré $3$, $\chi$ un caractère non principal de $G=\rm{Gal}(\mathbb{K}/\mathbb{Q})$ et $F \in \mathbb{Z}[X,Y]$ une forme binaire de degré $3$ irréductible sur $\mathbb{K}$. Nous avons, uniformément pour $y\geqslant 2$
\begin{equation}
\label{eq coro 1}
\sum\limits_{d>y}{(\chi*\chi^2)(d)\frac{\varrho_F^*(d)}{d^2}}\ll \frac{||F||^{\varepsilon}}{\log y}
\end{equation}
et
\begin{equation}
\label{eq coro 2}
\sum\limits_{(d_1,d_2)\notin [1,y]^2}{\chi(d_1)\chi^2(d_2)\frac{\varrho_F^*(d_1d_2)}{d_1^2d_2^2}}\ll \frac{||F||^{\varepsilon}}{\log y}{\rm .}
\end{equation}
\end{lemme}
\begin{proof}
La majoration \eqref{eq coro 1} se déduit directement des majorations \eqref{eq 2} et \eqref{eq 3} au moyen d'une intégration par parties. Pour démontrer la majoration \eqref{eq coro 2}, nous utilisons dans un premier temps la majoration \eqref{eq coro 1} pour majorer 
$$\sum\limits_{d>y^2}{(\chi*\chi^2)(d)\frac{\varrho_F^*(d)}{d^2}}{\rm .}$$
Il nous reste à majorer 
$$\sum\limits_{\substack{(d_1,d_2)\notin [1,y]^2 \\d_1d_2\leqslant y^2}}{\chi(d_1)\chi^2(d_2)\frac{\varrho_F^*(d_1d_2)}{d_1^2d_2^2}}{\rm .}$$
Majorons la contribution des couples $(d_1,d_2)$ tels que $d_1\leqslant y$. La contribution complémentaire est majorée de manière identique. Nous posons 
$$f_{d_1}(n):=\frac{\varrho_F^*(d_1n)}{\varrho_F^*(d_1)}{\rm .}$$
Il est facile de voir que la fonction $f_{d_1}$ est multiplicative. De plus, elle coïncide avec $\varrho_F^*(n)$ lorsque $n$ est premier avec $d_1$.\\ 

Déterminons une majoration de
\begin{equation}
\label{eq 13}
\sum\limits_{ y<d_2\leqslant y^2/d_1}{\chi^2(d_2)\frac{\varrho_F^*(d_1d_2)}{\varrho_F^*(d_1)d_2^2}}{\rm .}
\end{equation} 
Considérons la série de Dirichlet associée à la fonction $\chi^2(n)f_{d_1}(n)$, que nous noterons $F_{d_1}(s)$. Si nous notons $G(s)$ la série de Dirichlet associée à la fonction $\chi^2(n)\varrho_F^*(n)$, nous pouvons déduire l'existence d'une fonction $H_{d_1}$ admettant une écriture en produit eulérien absolument convergent sur le demi-plan \\
$\{s \in \mathbb{C} \mbox{ : }\mathcal{R}e(s)>1\}$ telle que $F_{d_1}(s)=H_{d_1}(s)G(s)$. De plus, ce produit eulérien s'écrit
\begin{equation}
H_{d_1}(s)=\prod\limits_{p\mid d_1}{\sum\limits_{\nu\geqslant 0}{\frac{h_{d_1}(p^{\nu})}{p^{\nu s}}}}{\rm .}
\end{equation}
 Nous pouvons déduire une majoration de la somme \eqref{eq 13} à partir d'une estimation des termes $h_{d_1}(p)=\chi^2(p)\Big(\frac{\varrho_F^*(p^{n+1})}{\varrho_F^*(p^{n})}-\varrho_F^*(p)\Big)$, où $n=v_p(d_1)$, pour tout $p \mid d_1$. Dans le cas où $p$ ne divise pas $F(1,0)\mathcal{D}(F)$, qui est non nul car $F$ est irréductible sur $\mathbb{K}$, nous pouvons utiliser le lemme 3.2 de \cite{B} afin d'obtenir 
\begin{align*}
|h_{d_1}(p)|&\leqslant p|1-\varrho_F^-(p)|+\varrho_F^-(p) \\
&\leqslant 2p+3{\rm .}
\end{align*}
Cela fournit
\begin{equation}
\label{rhod1}
\sum\limits_{ y<d_2\leqslant y^2/d_1}{\chi^2(d_2)\frac{\varrho_F^*(d_1d_2)}{\varrho_F^*(d_1)d_2^2}}=O\Big(||F||^{\varepsilon}\frac{d_1^2}{\phi(d_1)^2}\frac{1}{(\log y)^2}\Big){\rm .}
\end{equation}
Enfin, 
\begin{align*}
\sum\limits_{1\leqslant d_1\leqslant y}{\chi(d_1)\frac{\varrho_F^*(d_1)}{d_1^2}\sum\limits_{ y<d_2\leqslant y^2/d_1}{\chi^2(d_2)\frac{\varrho_F^*(d_1d_2)}{\varrho_F^*(d_1)d_2^2}}} &\ll \frac{||F||^{\varepsilon}}{(\log y)^2} \sum\limits_{1\leqslant d_1\leqslant y}{\frac{\varrho_F^*(d_1)}{\phi(d_1)^2}}\\
&\ll \frac{||F||^{2\varepsilon}}{\log y}{\rm ,}
\end{align*}
ce qui démontre la majoration \eqref{eq coro 2}.
\end{proof}
\subsection{Estimation de fonctions arithmétiques appliquées à des formes binaires}
 Avant d'énoncer le prochain résultat, établi dans \cite{B-T}, nous définissons une certaine classe de fonctions. 

 Lorsque $A_1\geqslant 1$, $A_2\geqslant 1$, et $\varepsilon>0$, nous notons $\mathcal{M}(A_1,A_2,\varepsilon)$ la classe des fonctions arithmétiques $f$ positives ou nulles satisfaisant la condition
$$f(ab)\leqslant \min(A_1^{\Omega(a)},A_2a^{\varepsilon})f(b)$$
pour tout $(a,b)=1$. Par ailleurs, pour $v\geqslant 1$ et $f$ une fonction arithmétique, nous définissons
\begin{equation}
E_F(v,f):=\prod\limits_{4<p\leqslant v}{\Big(1-\frac{\varrho_F^+(p)}{p^2}\Big)}\sum\limits_{1\leqslant s\leqslant v}{f(s)\frac{\varrho_F^+(s)}{s^2}}{\rm .}
\end{equation}
\begin{lemme}
\label{theo B}
Soit $F\in \mathbb{Z}[X,Y]$ une forme binaire primitive de degré au plus~$3$ irréductible sur $\mathbb{Q}$. Pour tout $\delta \in ]0,1[$, $A_1 \geqslant 1$, $A_2\geqslant 1$, il existe une constante $c_0$ telle que, uniformément sous les conditions 
\begin{align*}
0<\varepsilon\leqslant \delta/4000 {\rm ,}& &&f\in \mathcal{M}(A_1,A_2,\varepsilon){\rm ,}&& \min (x,y)\geqslant c_0\max(x,y,||F||)^{\delta}{\rm ,}
\end{align*}
on ait 
\begin{equation}
\sum\limits_{\substack{m\leqslant x \\ n\leqslant y}}{f(|F(m,n)|)}\ll xyE_F(x+y,f){\rm .}
\end{equation}
De plus, il existe une constante $C=C(A_1)>0$ telle que
\begin{equation}
E_F(v,f)\ll \prod\limits_{p\mid \mathcal{D}(F)}{\Big(1+\frac{1}{p}\Big)^C}\exp\Bigg(\sum\limits_{\substack{p\leqslant v \\ p\nmid \mathcal{D}(F)}}{\frac{\varrho_F^-(p)}{p}(f(p)-1)}\Bigg){\rm ,}
\end{equation}
où $\mathcal{D}(F)$ désigne le discriminant de $F$.
\end{lemme}

\begin{lemme}
\label{lemme B}
Soient $J \in \mathbb{Z}[X,Y]$ une forme binaire irréductible sur $\mathbb{Q}$ de degré au plus $3$, $A_1\geqslant 1$, $A_2\geqslant 1$, $\kappa>0$, $D\in \mathbb{N}^*$ et $M \in M_2(\mathbb{Z})$ de déterminant non nul. On suppose que la forme binaire
$$F(X,Y)=\frac{J(M(X,Y))}{D}$$
est à coefficients entiers. Il existe alors $\varepsilon_1>0$ tel que pour tout $0<\varepsilon<\varepsilon_1$, pour tout $f \in \mathcal{M}(A_1,A_2,\varepsilon)$ et uniformément sous la condition
$$v\geqslant ||M||^{\kappa}+\e^{\kappa D}$$
on ait
$$E_F(v,f)\ll ||F||^{\varepsilon}E_J(v,f){\rm ,} \;\;\;\;\;\; E_J(v,f)\ll\prod\limits_{p\leqslant v}{\Big(1+\frac{\rho_J^+(p)}{p}(f(p)-1)\Big)}{\rm .}$$
\end{lemme}
\subsection{Sur les fonctions Delta de Hooley généralisées}
 Pour $\mathbf{f}=(f_1,f_2)$ un couple de fonctions arithmétiques, $\mathbf{u}=(u_1,u_2) \in \mathbb{R}^2$, $\mathbf{v}=(v_1,v_2) \in [0,1]^2$ et $n\geqslant 1$, nous posons
$$\Delta_3(n,\mathbf{f},\mathbf{u},\mathbf{v}):=\sum\limits_{\substack{d_1d_2\mid n \\ \e^{u_i}<d_i\leqslant \e^{u_i+v_i}}}{f_1(d_1)f_2(d_2)}{\rm ,}$$
$$\Delta_3(n,\mathbf{f}):=\sup\limits_{\substack{\mathbf{u} \in \mathbb{R}^2 \\\mathbf{v} \in [0,1]^2}} |\Delta_3(n,\mathbf{f},\mathbf{u},\mathbf{v})|{\rm .}$$
Lorsque $\mathbf{f}=(1,1)$, nous obtenons la fonction de Hooley $\Delta_3$, étudiée notamment aux chapitres $6$ et $7$ de \cite{H-T}. Pour $A_1\geqslant 1$, nous désignons par $\mathcal{M}(A_1)$ l'ensemble $\cap_{\varepsilon >0}\cup_{A_2\geqslant 1} \mathcal{M}(A_1,A_2,\varepsilon)$. Pour tout $c>0$, nous notons $\mathcal{M}(A_1,c)$ le sous-ensemble de $\mathcal{M}(A_1)$ constitué des fonctions $g$ vérifiant
$$\sum\limits_{p\leqslant x}{g(p)}=\li(x)+O(x\e^{-c\sqrt{\log x}}) {\rm .}$$ 
Enfin, pour tout $\chi$, caractère de Dirichlet d'ordre $3$, nous notons $\mathcal{M}(A_1,\chi,c)$ le sous-ensemble de $\mathcal{M}(A_1,c)$ constitué des fonctions arithmétiques $g$ vérifiant 
\begin{align*}
&\sum\limits_{p\leqslant x}{\chi(p)g(p)}\ll x\e^{-c\sqrt{\log x}} {\rm ,}\\
&\sum\limits_{p\leqslant x}{\chi^2(p)g(p)}\ll x\e^{-c\sqrt{\log x}}{\rm .}
\end{align*}
\begin{remarque}
La fonction  $n\rightarrow~\Delta_3(n , \chi,\chi^2)$ appartient à $\mathcal{M}(3)$.
\end{remarque}
Définissons 
$$\mathcal{L}(x):=\exp{\sqrt{\log_2 x\log_3 x}}$$
et
\begin{equation}
\label{rho}
\rho:=\frac{1}{2\pi}\int_{0}^{2\pi}{\max(1,|1+\e^{it}|^2){\rm d}t}-2 \approx 0.21800 {\rm .}
\end{equation}
Nous rappelons dans un premier temps le résultat fondamental obtenu par Hall et Tenenbaum, correspondant aux théorèmes 70 et 71 de \cite{H-T}.
\begin{theoreme}
\label{theo T}
Soient $A_1\geqslant 1$, $c>0$ et $g \in \mathcal{M}(A_1,c)$. Pour tout $y>0$, lorsque $x\geqslant 10$, il existe une constante $\alpha=\alpha(g,y)>0$ telle que l'on ait
\begin{equation}
\sum\limits_{n\leqslant x}{g(n)y^{\omega(n)}\Delta_3(n)}\ll x(\log x)^{\max\{y-1,3y-3\}}\mathcal{L}(x)^{\alpha}{\rm .}
\end{equation}
\end{theoreme}
Le résultat suivant, correspondant au théorème 1.1 de \cite{L}, fournit une estimation des compensations dues aux oscillations des caractères $\chi$ et $\chi^2$ dans des intervalles du type $]\e^{u_1},\e^{u_1+1}]\times ]\e^{u_2},\e^{u_2+1}]$.
\begin{theoreme}
\label{theo L}
Soient $\chi$ un caractère de Dirichlet d'ordre $3$, $A_1\geqslant ~1$, $c>0$ et $g \in \mathcal{M}(A_1,\chi,c)$. Pour tout $y>0$, lorsque $x\geqslant 10$, il existe une constante $\alpha=\alpha(g,y)>0$ telle que l'on ait
\begin{equation}
\sum\limits_{n\leqslant x}{g(n)y^{\omega(n)}\Delta_3(n,\chi,\chi^2)^2}\ll x(\log x)^{\max\{y-1,(\rho+2)y-2,3y-3\}}\mathcal{L}(x)^{\alpha}{\rm .}
\end{equation}
\end{theoreme}
Des compensations des oscillations seraient d'ordre statistique si l'exposant de $\log x$ valait $\max\{y-1,3y-3\}$, autrement dit, si nous avions $\rho=0$. La valeur de $\rho$ obtenue permet néanmoins d'appliquer ce résultat au problème de comptage du Théorème \ref{theo 1}. 

\section{Simplification du problème}
\subsection{Une bonne approximation}
Pour estimer $Q(F,\xi,\mathcal{R})$, nous constatons que nous pouvons le réécrire sous la forme suivante
\begin{align*}
Q(F,\xi,\mathcal{R})&=\sum\limits_{(d_1,d_2) \in (\mathbb{N}^*)^2}{\chi(d_1)\chi^2(d_2)\sum\limits_{\substack{\bfx \in \mathbb{Z}^2\cap\mathcal{R}(\xi) \\ d_1d_2\mid F(\bfx)}}{1}} \\
&=\sum\limits_{(d_1,d_2) \in (\mathbb{N}^*)^2}{\chi(d_1)\chi^2(d_2)|\Lambda(d_1d_2,F)\cap\mathcal{R}(\xi)|}{\rm ,}
\end{align*}
où $\Lambda(s,F)$ est défini en \eqref{lambda}. Nous aurons donc besoin d'une bonne approximation de la quantité $|\Lambda(d_1d_2,F)\cap\mathcal{R}(\xi)|$ pour démontrer le Théorème~\ref{theo 1}. Pour des raisons techniques qui apparaitront par la suite, nous remplaçons~$\mathcal{R}(\xi)$ par $\mathcal{D}_q\cap \mathcal{R}(\xi)$ où $\mathcal{D}_q$ est défini par 
\begin{equation}
\label{Dq}
\mathcal{D}_q:=\{(m,n) \in \mathbb{Z}^2 \mbox{ : } (m,q)=1\}{\rm .}
\end{equation}
Pour $y_1, y_2 \geqslant 1$, $\sigma,\xi,\vartheta>0$ et $F \in \mathbb{Z}[X,Y]$ une forme binaire de degré $3$, nous posons
$$\Phi(\xi,y_1,y_2,F,\sigma,\vartheta):=\!\!\!\!\!\sum\limits_{\substack{1\leqslant d_1\leqslant y_1\\ 1\leqslant d_2\leqslant y_2 \\(q,d_1d_2)=1}}{\!\!\!\sup\limits_{\mathcal{R}}{\Big|| \Lambda(d_1d_2,F)\cap \mathcal{D}_q \cap \mathcal{R}(\xi)|-{\rm vol}(\mathcal{R})\xi^2\frac{\varphi(q)\varrho_F^+(d_1d_2)}{qd_1^2d_2^2}\Big|}} $$
où le $\sup$ est pris sur l'ensemble des domaines $\mathcal{R}$ vérifiant les hypothèses (H1), (H2) et (H3).
\begin{lemme}
\label{lemme 2}
Soient $\kappa>0$, $\varepsilon>0$, $\sigma>0$, $\vartheta>0$, $J$ une forme binaire de degré $3$ irréductible sur $\mathbb{Q}$, $F(\bfx)=J(M\bfx)/D \in \mathbb{Z}[X,Y]$, avec $M \in M_2(\mathbb{Z})$ de déterminant non nul et $D\in \mathbb{N}^*$. Sous les conditions

$$y_1,y_2\geqslant 2{\rm ,} \;\;\;\;\; \xi\geqslant ||M||^{\kappa}+\e^{\kappa D}{\rm ,} \;\;\;\;\; 1/\sqrt{\xi}\leqslant \sigma\leqslant \xi^{3/2}{\rm ,} \;\;\;\;\; 1/\sqrt{\xi}\leqslant \vartheta\leqslant \xi^{3/2}{\rm ,}$$
nous avons
\begin{equation}
\label{maj 3}
\Phi(\xi,y_1,y_2,F,\sigma,\vartheta)\ll ||F||^{\varepsilon}\big(\sigma\xi\sqrt{y_1y_2}+y_1y_2\big)
\mathcal{L}(\sigma\xi)^{\sqrt{3}+\varepsilon}{\rm .}
\end{equation}
\end{lemme}

\begin{proof}
Le début de la démonstration est identique à celui du lemme~5.2 de \cite{B}. Nous pouvons supposer que la forme $F$ est primitive, la majoration recherchée découle alors d'une majoration de la quantité suivante

\begin{equation}
\Phi^*:= \sum\limits_{\substack{1\leqslant d_1\leqslant y_1 \\ 1\leqslant d_2\leqslant y_2 \\(q,d_1d_2)=1}}{\sup\limits_{\mathcal{R}}{\Big|| \Lambda^*(d_1d_2,F)\cap \mathcal{D}_q \cap \mathcal{R}(\xi)|-{\rm vol}(\mathcal{R})\xi^2\frac{\varphi(q)\varrho_F^*(d_1d_2)}{qd_1^2d_2^2}\Big|}} {\rm ,}
\end{equation}
où $\Lambda^*$ est défini en \eqref{lambda*}. De même  que dans la démonstration du lemme 5.2 de \cite{B}, nous utilisons la formule d'inversion de Möbius et l'approximation du nombre de points d'un réseau dans un domaine convexe de $\mathbb{R}^2$ pour obtenir
\begin{equation}
\label{34}
\Phi^*\ll \sigma\xi\Phi_1^*+\Phi_2^*
\end{equation}
où 
$$\Phi_1^*:=\sum\limits_{\substack{b_1\leqslant y_1 \\ b_2\leqslant y_2}}{\sum\limits_{\substack{t_1\leqslant y_1/b_1\\t_2\leqslant y_2/b_2}}{\sum\limits_{\mathcal{A}\in \mathcal{U}_F(b_1b_2t_1t_2)}{\frac{1}{||v_{t_1t_2}(\mathcal{A})||}}}}$$
et
$$\Phi_2^*:=\sum\limits_{\substack{s_1\leqslant y_1\\ s_2\leqslant y_2}}{\frac{\varrho_F^*(s_1s_2)}{\phi(s_1s_2)}\sum\limits_{\substack{b_1\leqslant y_1 \\ b_2\leqslant y_2\\ b_1\mid s_1 \\b_2\mid s_2}}{\min \Big\{1,\frac{\sigma^2 \xi^2}{b_1b_2s_1s_2}\Big\}}} {\rm .}$$
Dans $\Phi_1^*$, pour $s \in \mathbb{N^*}$, $\mathcal{U}_F(s)$ désigne l'ensemble des classes d'équivalence de $\Lambda^*(s,F)$ défini par la relation $\bfx\sim \bfy$ si et seulement si il existe $\lambda \in \mathbb{Z}$ tel que $\bfx\equiv\lambda\bfy \bmod s$. Cet ensemble est de cardinal $\varrho_F^*(s)/\phi(s)$. Pour $\mathcal{A} \in \mathcal{U}_F(s)$, $t \mid s$ et $\bfx \in \mathcal{A}$, nous désignons par $\mathcal{A}_t$ l'ensemble
$$\mathcal{A}_t:=\{\mathbf{y} \in \mathbb{Z}^2\mbox{ : } \exists \lambda \in \mathbb{Z}{\rm , }\bfy\equiv \lambda \bfx \bmod t\}{\rm ,}$$
puis, nous définissons $v_t(\mathcal{A})$ comme un vecteur minimal non nul de $\mathcal{A}_t$. Ce vecteur vérifie alors 
$$||v_t(\mathcal{A})||\leqslant \sqrt{2t}{\rm .}$$
Nous réutilisons de nouveau la formule (5.29) de \cite{B} pour majorer $\Phi_1^*$. Lorsque $||v_{t_1t_2}||/\sqrt{2y_1y_2/b_1b_2} \in ]1/2^{j+1}, 1/2^j]$, nous avons $t_1t_2>y_1y_2/(b_1b_22^{2(j+1)})$. Nous pouvons donc écrire
$$\Phi_1^*\ll\sum\limits_{\substack{b_1\leqslant y_1 \\ b_2\leqslant y_2}}{\frac{1}{b_1b_2}\sum\limits_{j\geqslant 0}{\sum\limits_{\substack{y_1/(b_12^{2(j+1)})<t_1\leqslant y_1/b_1 \\y_2/(b_22^{2(j+1)})<t_2\leqslant y_2/b_2}}{\frac{2^j}{\sqrt{y_1y_2/b_1b_2}}\sum\limits_{\substack{v \in (\mathbb{Z}^2)^* \\ ||v||\leqslant \sqrt{2y_1y_2/b_1b_2}/2^j \\ t_1t_2 \mid T(v)}}{1}}}}{\rm .}$$
Nous appliquons alors le Lemme \ref{theo B} pour obtenir
\begin{align}
\label{35}
\begin{split}
\Phi_1^*&\ll\sum\limits_{\substack{b_1\leqslant y_1 \\ b_2\leqslant y_2}}{\frac{1}{b_1b_2}\sum\limits_{j\geqslant 0}{\frac{(j+1)^22^j}{\sqrt{y_1y_2/b_1b_2}}\sum\limits_{\substack{v \in (\mathbb{Z}^2)^* \\ ||v||\leqslant \sqrt{2y_1y_2/b_1b_2}/2^j }}{\Delta_3(F(v),1)}}} \\
&\ll ||F||^{\varepsilon}\sqrt{y_1y_2} \mathcal{L}(\sigma\xi)^{\sqrt{3}+\varepsilon} {\rm .}
\end{split}
\end{align}
Majorons à présent $\Phi_2^*$. La contribution $\Phi_{21}^*$ à $\Phi_2^*$ des entiers $s_1$, $s_2$ tels que $b_1b_2s_1s_2\geqslant \sigma^2\xi^2$ vérifie
\begin{equation}
\label{36}
\Phi_{21}^*\leqslant \sigma^2\xi^2\sum\limits_{s\leqslant y_1y_2}{\frac{\varrho_F^*(s)\tau(s)}{\varphi(s)}\sum\limits_{\substack{b\leqslant y_1y_2 \\ b \mid s \\bs>\sigma^2\xi^2}}{\frac{\tau(b)}{bs}}}{\rm .}
\end{equation}
Cette majoration s'obtient en posant $b=b_1b_2$ et $s=s_1s_2$  dans la somme définissant $\Phi_{21}^*$. Nous pouvons alors écrire
\begin{align}
\label{36.5}
\nonumber
\Phi_{21}^* &\ll_{\varepsilon} ||F||^{\varepsilon/3}\sigma^2\xi^2\sum\limits_{b>(\sigma \xi)^2/(y_1y_2)}{\frac{\varrho_F^*(b)\tau^2(b)}{\varphi(b)b^2}\sum\limits_{\sigma^2\xi^2/b<t\leqslant y_1y_2/b}{\frac{\varrho_F^*(t)\tau(t)}{\varphi(t)}}}\\
\nonumber
\Phi_{21}^* &\ll_{\varepsilon} ||F||^{2\varepsilon/3}\sigma^2\xi^2\sum\limits_{b>(\sigma \xi)^2/(y_1y_2)}{\frac{\varrho_F^*(b)\tau^2(b)}{\varphi(b)b^2}\log^2\Big(2+\frac{y_1y_2b}{\sigma^2\xi^2}\Big)}\\
\Phi_{21}^*& \ll_{\varepsilon}||F||^{\varepsilon}\min(\sigma^2\xi^2,y_1y_2)\log^2\Big(2+\frac{y_1y_2}{\sigma^2\xi^2}\Big){\rm .}
\end{align}
Par ailleurs, la contribution $\Phi_{22}^*$ à $\Phi_2^*$ des entiers $s_1$, $s_2$ tels que $b_1b_2s_1s_2\leqslant \sigma^2\xi^2$ vérifie
\begin{align*}
\Phi_{22}^* & \leqslant \sum\limits_{\substack{b_1\leqslant y_1 \\ b_2\leqslant y_2}}{\sum\limits_{\substack{t_1\leqslant y_1/b_1 \\ t_2\leqslant y_2/b_2 \\ t_1t_2\leqslant \sigma^2\xi^2/(b_1^2b_2^2)}}{\frac{\varrho_F^*(b_1b_2t_1t_2)}{\varphi(b_1b_2t_1t_2)}}} \\
& \ll_{\varepsilon} ||F||^{\varepsilon/5} \sum\limits_{\substack{b_1\leqslant y_1 \\ b_2\leqslant y_2}}{\frac{\varrho_F^*(b_1)\varrho_F^*(b_2)}{\varphi(b_1)\varphi(b_2)}{\sum\limits_{t_1\leqslant y_1/b_1}{\frac{\varrho_F^*(t_1)}{\varphi(t_1)}{\sum\limits_{t_2\leqslant \min\{(y_2/b_2, \sigma^2\xi^2/(b_1^2b_2^2t_1)\}}{\frac{\varrho_F^*(t_2)}{\varphi(t_2)}}}}}} \\
& \ll_{\varepsilon} ||F||^{2\varepsilon/5}\sum\limits_{\substack{b_1\leqslant y_1 \\ b_2\leqslant y_2}}{\frac{\varrho_F^*(b_1)\varrho_F^*(b_2)}{\varphi(b_1)\varphi(b_2)}{\sum\limits_{t_1\leqslant y_1/b_1}{\frac{\varrho_F^*(t_1)}{\varphi(t_1)b_2}{\min\Big(\frac{\sigma^2\xi^2}{t_1b_1^2b_2},y_2\Big)}}}}\\
& \ll_{\varepsilon} ||F||^{2\varepsilon/5}\sum\limits_{b_1\leqslant y_1 }{\frac{\varrho_F^*(b_1)}{\varphi(b_1)}\sum\limits_{t_1\leqslant y_1/b_1}{\frac{\varrho_F^*(t_1)}{\varphi(t_1)}\sum\limits_{b_2\leqslant y_2}{\frac{\varrho_F^*(b_2)}{\varphi(b_2)}\min\Big(\frac{\sigma^2\xi^2}{t_1b_1^2b_2},y_2\Big)}}}
\end{align*}
Nous utilisons la majoration de $\Phi_{22}^*$ déterminée dans \cite{B}, en remplaçant le terme $\sigma^2\xi^2$ par $\sigma^2\xi^2/(t_1b_1^2)$, pour majorer la somme intérieure. Nous obtenons ainsi
$$\Phi_{22}^* \ll_{\varepsilon} ||F||^{3\varepsilon/5}y_2\sum\limits_{b_1\leqslant y_1 }{\frac{\varrho_F^*(b_1)}{\varphi(b_1)}\sum\limits_{t_1\leqslant y_1/b_1}{\frac{\varrho_F^*(t_1)}{\varphi(t_1)}\min\Big(\frac{\sigma^2\xi^2}{t_1b_1^2y_2},1\Big)\log\Big(2+\frac{\sigma^2\xi^2}{t_1b_1^2y_2}\Big)}}{\rm .}$$
Notons $\beta$ le majorant de $\Phi_{22}^*$ ci dessus. Nous souhaitons alors utiliser la majoration de $\Phi_2^*$ déterminée dans \cite{B}, en remplaçant $\sigma^2\xi^2$ par $\sigma^2\xi^2/y_2$, cependant, nous devons tenir compte du terme $\log\Big(2+\frac{\sigma^2\xi^2}{t_1b_1^2y_2}\Big)$ apparaissant dans la somme. En modifiant légèrement la démonstration de cette majoration, nous parvenons néanmoins au résultat souhaité. En effet, lorsque $t_1b_1^2\geqslant \sigma^2\xi^2/y_2$, nous avons $\log\Big(2+\frac{\sigma^2\xi^2}{t_1b_1^2y_2}\Big) \ll 1$. Ainsi, nous pouvons utiliser la majoration de $\Phi_{21}^*$ de \cite{B} pour majorer la contribution $\beta_1$ à $\beta$ des entiers $b_1$ et $t_1$ tels que $t_1b_1^2\geqslant \sigma^2\xi^2/y_2$. Nous obtenons
\begin{equation}
\label{37}
\beta_1\ll_{\varepsilon} ||F||^{\varepsilon}\sigma\xi\min(\sigma\xi,\sqrt{y_1y_2})\log\Big(2+\frac{y_1y_2}{\sigma^2\xi^2}\Big){\rm .}
\end{equation}
Pour majorer la contribution $\beta_2$ à $\beta$ des entiers $b_1$, $t_1$ tels que $t_1b_1^2\leqslant \sigma^2\xi^2/y_2$, nous écrivons
\begin{align}
\label{38}
\begin{split}
\beta_2& \ll_{\varepsilon} ||F||^{3\varepsilon/5}y_2\sum\limits_{b_1\leqslant y_1 }{\frac{\varrho_F^*(b_1)}{\varphi(b_1)}\sum\limits_{t_1\leqslant \min(y_1/b_1,\sigma^2\xi^2/(b_1^2y_2) }{\frac{\varrho_F^*(t_1)}{\varphi(t_1)}\log\Big(\frac{\sigma^2\xi^2}{t_1b_1^2y_2}\Big)}} \\
& \ll_{\varepsilon} ||F||^{4\varepsilon/5}y_2\sum\limits_{b_1\leqslant y_1 }{\frac{\varrho_F^*(b_1)}{\varphi(b_1)b_1}\min\Big(y_1,\frac{\sigma^2\xi^2}{b_1y_2}\Big)\log\Big(2+\frac{\sigma^2\xi^2}{y_1b_1y_2}\Big)} \\
& \ll_{\varepsilon} ||F||^{\varepsilon}y_2\min\Big(y_1,\frac{\sigma^2\xi^2}{y_2}\Big)\log\Big(2+\frac{\sigma^2\xi^2}{y_1y_2}\Big)
\end{split}
\end{align}
Les inégalités \eqref{37} et \eqref{38} fournissent
\begin{equation}
\label{39}
\beta\ll_{\varepsilon} ||F||^{\varepsilon}\sigma\xi\min(\sqrt{y_1y_2},\sigma\xi)\log_2\xi {\rm .}
\end{equation}
La majoration de $\Phi_2^*$ se déduit alors des majorations \eqref{36.5} et \eqref{39}.
\begin{equation}
\label{40}
\Phi_2^*\ll_{\varepsilon} ||F||^{\varepsilon}\sigma\xi\min(\sqrt{y_1y_2},\sigma\xi)\Big\{\log_2\xi+\log\Big(2+\frac{y_1y_2}{\sigma^2\xi^2}\Big)\Big\}{\rm .}
\end{equation} 
En reportant les équations \eqref{35} et \eqref{40} dans la formule \eqref{34}, nous obtenons la majoration \eqref{maj 3}.
\end{proof}
Nous énonçons deux corollaires de ce résultat.
\begin{corollaire}
\label{coro 3.3}
Soient $y\geqslant 2$, $\varepsilon$, $\sigma$, $\xi$ vérifiant les mêmes conditions que celles du Lemme \ref{lemme 2}. Nous avons uniformément pour $u \in \mathbb{R^+}$

\begin{align}
\label{maj coro}
\begin{split}
\sum\limits_{\substack{1\leqslant d_1d_2\leqslant y \\ \e^u<d_2\leqslant \e^{u+1} \\(q,d_1d_2)=1}}{\sup\limits_{\mathcal{R}}{\Big|| \Lambda(d_1d_2,F)\cap \mathcal{D}_q \cap \mathcal{R}(\xi)|-{\rm vol}(\mathcal{R})\xi^2\frac{\varphi(q)\varrho_F^+(d_1d_2)}{qd_1^2d_2^2}\Big|}} \\
 \ll_{\varepsilon}  ||F||^{\varepsilon}\big((\sigma+\vartheta)\xi\sqrt{y}+y\big)
\mathcal{L}(\sigma\xi)^{\sqrt{3}+\varepsilon}{\rm .}
\end{split}
\end{align}
\end{corollaire}

\begin{proof}
Les conditions $d_1d_2\leqslant y$ et $\e^u<d_2\leqslant \e^{u+1}$ impliquent $d_2\leqslant ~\e^{u+1}$ et $d_1\leqslant y\e^{-u}$. Il nous suffit d'appliquer le Lemme \ref{lemme 2} avec ces valeurs pour $y_1$ et $y_2$ afin d'obtenir le résultat.
\end{proof}
Pour le second corollaire, nous posons
\begin{align}
\label{eq phi}
\begin{split}
\phi&:=\phi(\xi,y,F,\sigma,\vartheta)\\
&:=\sum\limits_{\substack{1\leqslant d_1d_2\leqslant y \\(q,d_1d_2)=1}}{\sup\limits_{\mathcal{R}}{\Big|| \Lambda(d_1d_2,F)\cap \mathcal{D}_q \cap \mathcal{R}(\xi)|-{\rm vol}(\mathcal{R})\xi^2\frac{\varphi(q)\varrho_F^+(d_1d_2)}{qd_1^2d_2^2}\Big|}}{\rm .} 
\end{split}
\end{align}
\begin{corollaire}
\label{lemme 1}
Soient $A>0$, $\kappa>0$, $\varepsilon>0$, $\sigma>0$, $\vartheta>0$, $J$ une forme binaire de degré $3$ irréductible sur $\mathbb{Q}$ et $F(\bfx)=J(M\bfx)/D \in \mathbb{Z}[X,Y]$, avec $M \in M_2(\mathbb{Z})$ de déterminant non nul. Sous les conditions
$$2\leqslant y \leqslant \xi^A{\rm ,} \;\;\;\;\; \xi\geqslant ||M||^{\kappa}+\e^{\kappa D}{\rm ,} \;\;\;\;\; 1/\sqrt{\xi}\leqslant \sigma\leqslant \xi^{3/2}{\rm ,} \;\;\;\;\; 1/\sqrt{\xi}\leqslant \vartheta\leqslant \xi^{3/2}{\rm ,}$$
nous avons
\begin{equation}
\label{maj}
\phi(\xi,y,F,\sigma,\vartheta)\ll_{\varepsilon,A} ||F||^{\varepsilon}\big(\sigma\xi\sqrt{y}+y\big)\mathcal{L}(\sigma\xi)^{\sqrt{2}+\varepsilon}\log \xi{\rm .}
\end{equation}
\end{corollaire}
\begin{proof}
Nous pouvons réécrire $\phi$ sous la forme suivante
$$\phi= \sum\limits_{k=0}^{[\log y]}{\sum\limits_{\substack{1\leqslant d_1d_2\leqslant y \\\e^k\leqslant d_2< \e^{k+1} \\(q,d_1d_2)=1}}{\sup\limits_{\mathcal{R}}{\Big|| \Lambda(d_1d_2,F)\cap \mathcal{D}_q \cap \mathcal{R}(\xi)|-{\rm vol}(\mathcal{R})\xi^2\frac{\varphi(q)\varrho_F^+(d_1d_2)}{qd_1^2d_2^2}\Big|}}} {\rm .}$$
Nous utilisons le corollaire \ref{coro 3.3} pour majorer la somme intérieure uniformément en $k$, nous obtenons ainsi
\begin{align*}
\phi&\ll_{\varepsilon} \sum\limits_{k=0}^{[\log y]}{||F||^{\varepsilon}\big((\sigma+\vartheta)\xi\sqrt{y}+y\big)
\mathcal{L}(\sigma\xi)^{\sqrt{3}+\varepsilon}} \\
&\ll_{\varepsilon}  ||F||^{\varepsilon}\big((\sigma+\vartheta)\xi\sqrt{y}+y\big)
\mathcal{L}(\sigma\xi)^{\sqrt{3}+\varepsilon}\log y \\
&\ll_{\varepsilon,A}  ||F||^{\varepsilon}\big((\sigma+\vartheta)\xi\sqrt{y}+y\big)
\mathcal{L}(\sigma\xi)^{\sqrt{3}+\varepsilon}\log \xi{\rm .}
\end{align*}
\end{proof}
\subsection{Paramétrisation de la somme \eqref{eq Q}} L'objectif dans cette section, est de paramétrer l'ensemble des couples $(m,n)$ pour lesquels $r_3\big(F((m,n))\big)\neq 0$, où $r_3$ est défini en \eqref{r}. \\

Notons que pour $d \mid q^{\infty}$ et $n \in \mathbb{Z}$, nous avons
$$r_3(dn)=r_3(n){\rm }$$
 \\

Nous pouvons, dans un premier temps, nous ramener au cas où $(m,n,q)=~~1$. En effet, comme le polynôme $F$ est homogène, nous pouvons écrire
\begin{equation}
\label{eq Q'}
Q(F,\xi,\mathcal{R})=\sum\limits_{d \mid q^{\infty}}{Q_1\Big(F,\frac{\xi}{d},\mathcal{R}\Big)}
\end{equation}
où
$$Q_1(F,\xi,\mathcal{R}):=\sum\limits_{\substack{\bfx \in \mathcal{R}(\xi)\cap \mathbb{Z}^2 \\ (q,\bfx)=1}}{r_3\big(F(\bfx)\big)}{\rm .}$$

Pour $\bfx \in \mathbb{Z}^2$ tel que $(q,\bfx)=1$, si nous posons $\bfx:=(m,n)$, il existe un unique entier $d_1 \mid q^{\infty}$ tel que $m=d_1m_1$ et $(m_1,q)=1$. La condition $(q,m,n)=1$ est alors équivalente à la condition $(d_1,n)=1$. De plus, si nous notons $d_2=(q^{\infty},F(\bfx))$  alors $r_3(F(\bfx))\neq 0$ implique qu'il existe $\alpha_1 \in (\mathbb{Z}/q\mathbb{Z})^{\times}$ vérifiant $\chi(\alpha_1)=1$ tel que $F(\bfx)/d_2\equiv \alpha_1 \bmod q$. Nous rappelons que $G_1={\rm Ker}(\chi) \subset (\mathbb{Z}/q\mathbb{Z})^{\times}$ et nous constatons de suite qu'il contient le groupe des cubes de $(\mathbb{Z}/q\mathbb{Z})^{\times}$. En posant $\alpha=\alpha_1m_1^{-3} \in(\mathbb{Z}/q\mathbb{Z})^{\times} $, nous avons $\alpha \in G_1$ et  $F(\bfx)/d_2\equiv m_1^3\alpha \bmod q$.\\ 

Pour $\alpha \in G_1$, $d_1\mid q^{\infty}$ et $d_2\mid q^{\infty}$, nous notons $W_{\alpha,d_1,d_2}$ l'ensemble des $\beta \in \mathbb{Z}/d_2q\mathbb{Z}$ premiers avec $d_1$ (Les entiers $d_1$ et $d_2$ étant des diviseurs de $q^{\infty}$, cette condition de primalité est bien définie.), assimilés à un système de représentants dans $[1,d_2q]$, vérifiant $F(d_1,\beta)\equiv \alpha d_2\bmod d_2q$. Nous pouvons alors écrire la congruence suivante
$$\bfx\equiv m_1(d_1,\beta)\bmod d_2q {\rm .}$$
Cela signifie qu'il existe un entier $n_1 \in \mathbb{Z}^2$ tel que 
$$\bfx=(d_1m_1,\beta m_1+d_2q n_1)=U_{\beta,d_1,d_2}(m_1,n_1)$$
où nous avons posé
 $$U_{\beta,d_1,d_2}:=\begin{pmatrix}
 d_1 & 0 \\
 \beta & d_2q
 \end{pmatrix}{\rm .}$$
 Notons que cette matrice est de déterminant $d_1d_2q$.\\
 Ainsi, un couple $\bfx$ contribuant à la somme \eqref{1} peut être déterminé par cinq paramètres $(d_1,\beta,d_2,m_1,n_1)$, il se trouve que cette détermination est unique, nous obtenons ainsi l'égalité suivante sur $Q_1$. 
\begin{equation}
\label{eq Q1}
Q_1(F,\xi,\mathcal{R})=\sum\limits_{d_1 \mid q^{\infty}}{\sum\limits_{\alpha \in G_1}{\sum\limits_{d_2 \mid q^{\infty}}{\sum\limits_{\beta \in W_{\alpha,d_1,d_2}}{Q_2(F_{\beta,d_1,d_2},\xi,\mathcal{R}_{\beta,d_1,d_2})}}}}
\end{equation}
où nous avons posé
\begin{align*}
F_{\beta,d_1,d_2}(m,n)&:=\frac{F(U_{\beta,d_1,d_2}(m,n))}{d_1}{\rm ,}\\
\mathcal{R}_{\beta,d_1,d_2}&:=\{\bfx \in \mathbb{R}^2 \mbox{ : } U_{\beta,d_1,d_2}(\bfx) \in \mathcal{R}\}{\rm ,}\\
Q_2(F,\xi,\mathcal{R})&:=\sum\limits_{(m,n) \in \mathcal{D}_q\cap \mathcal{R}(\xi)  }{r_3(F(m,n))}
\end{align*}
où $\mathcal{D}_q$ est défini en \eqref{Dq}.

\subsection{Estimation de la vitesse de convergence des sommes dans la formule \eqref{eq Q1}}
Afin d'exploiter la formule \eqref{eq Q1} dans la démonstration du Théorème \ref{theo 1}, nous avons besoin d'une estimation de la vitesse de convergence des sommes sur $d_1$ et sur $d_2$. Pour cela, nous énonçons les lemmes suivants.

\begin{lemme}
\label{chinois}
Soient $d_2\mid q^{\infty}$ et $\alpha \in G_1$. Pour tout $d_1 \mid q^{\infty}$, le cardinal de $W_{\alpha,d_1,d_2}$ ne dépend que de $d_3:= (d_1,d_2q)$.
\end{lemme}

\begin{proof}
D'après le théorème chinois, il existe un couple $(u,v) \in \mathbb{Z}^2$ tel que 
$$d_1u+d_2qv=d_3{\rm .}$$
Il est par ailleurs possible de choisir $u$ tel que $(u,d_2q)=1$. Un tel $u$ fournit alors une bijection entre $W_{\alpha,d_1,d_2}$ et $W_{\alpha,d_3,d_2}$, d'où le résultat.
\end{proof}

\begin{lemme}
\label{triv}
Soit $d_2 \mid q^{\infty}$ et $d_3 \mid d_2q$. Pour tout $\alpha>0$, nous avons
$$\sum\limits_{\substack{d_2 \mid q^{\infty} \\ (d_2,d_1q)=d_3}}{\frac{1}{d_2^{\alpha}}}\ll_{q,\alpha} \frac{1}{d_3^{\alpha}}{\rm .}$$
\end{lemme}
\begin{lemme}
\label{conv}
La contribution des entiers $d_2\geqslant \log(\xi)^{\frac{5\log p_{\omega(q)}}{2\log p_1}}$ à la somme \eqref{eq Q1} peut être incluse dans le terme d'erreur de \eqref{1}.
\end{lemme}
\begin{proof}
Notons dans un premier temps que pour $\alpha \in G_1$, $d_2 \mid q^{\infty}$ et $d_3 \mid d_2q$ fixés, pour $\beta \in W_{\alpha,d_3,d_2}$, il existe $\varphi(d_2q/d_3)$ éléments $\bfx=(d_3m_1,n) \in (\mathbb{Z}/(d_2q)\mathbb{Z})^2$ tels que $n\equiv m_1\beta \bmod (d_1q)$ et $(m_1,d_2q/d_3)=1$. Ces éléments vérifient alors l'équation $F(\bfx)\equiv m_1^3\alpha d_2\bmod (d_2q)$, en particulier, ils vérifient $F(\bfx)\equiv 0\bmod d_2$. Cela fournit l'inégalité
$$\sum\limits_{d_3 \mid d_2q}{\varphi\Big(\frac{d_2q}{d_3}\Big)|W_{\alpha,d_3,d_2}|}\leqslant q^2 \varrho_F^+(d_2)\ll_F 1{\rm .}$$
Ainsi, 
\begin{equation}
\label{W}
\sum\limits_{d_3 \mid d_2q}{\frac{|W_{\alpha,d_3,d_2}|}{d_3}}\leqslant \frac{q^2\varrho_F^+(d_2)}{\phi(q)d_2}\ll_{F,q} 1{\rm .}
\end{equation}
 
Pour conclure la démonstration, nous notons que nous avons nécessairement $d_2\leqslant \xi^{O(1)}$, ce qui permet de dominer toutes les valuations p-adiques de $d_2$ par $\log \xi$. Lorsque $\xi/d_2\leqslant \xi^{\varepsilon/3}$, une majoration triviale fournit
$$\sum\limits_{\bfx \in \mathbb{Z}^2\cap \mathcal{R}_{\beta,d_1,d_2}(\xi)}{\tau_3(F_{\beta,d_1,d_2}(\bfx))} \ll \frac{(\sigma\xi)^{1+\varepsilon}}{d_1}$$
ce qui est un terme d'erreur acceptable pour \eqref{1} d'après \eqref{W}, le Lemme \ref{triv} et le Lemme \ref{chinois}. Dans le cas contraire, nous sommes en mesure d'appliquer le théorème 4 de \cite{He} pour évaluer la sommation relative à $m$ de la somme ci-dessus. Il vient 
$$\sum\limits_{\bfx \in \mathbb{Z}^2\cap \mathcal{R}_{\beta,d_1,d_2}(\xi)}{\tau_3(F_{\beta,d_1,d_2}(\bfx))} \ll \frac{||F||^{\varepsilon}\sigma^2\xi^2 \log(\xi)^2}{(d_1d_2)^{1-\varepsilon}}{\rm .}$$
En reportant dans \eqref{eq Q1}, nous obtenons que la somme portant sur les entiers $d_2\geqslant \log(\xi)^{\frac{5\log p_{\omega(q)}}{2\log p_1}}$ est
$$\ll \sum\limits_{\substack{d_2 \mid q^{\infty}\\d_2\geqslant (\log \xi)^B}}{\sum\limits_{\alpha \in G_1}{\sum\limits_{d_1 \mid q^{\infty}}{\frac{|W_{\alpha,d_1,d_2}|}{(d_1d_2)^{1-\varepsilon}}}}}{\rm .}$$
Nous utilisons ensuite le Lemme $\ref{triv}$ pour majorer la somme sur $d_1$. Ainsi, la somme qui nous intéresse est
$$\ll_q \sum\limits_{\substack{d_2 \mid q^{\infty}\\d_2\geqslant (\log \xi)^A}}{\sum\limits_{\alpha \in G_1}{\sum\limits_{d_3 \mid d_2q}{\frac{|W_{\alpha,d_3,d_2}|}{(d_2d_3)^{1-\varepsilon}}}}}{\rm ,}$$
où $A:=\frac{5\log p_{\omega(q)}}{2\log p_1}$. Nous majorons alors $d_3^{\varepsilon}$ par $(d_2q)^{\varepsilon}$ et nous utilisons la majoration \eqref{W}. La contribution à la somme \eqref{eq Q1} des entiers $d_2\geqslant  (\log \xi)^A$ est donc
$$\ll_q \sum\limits_{\substack{d_2 \mid q^{\infty}\\d_2\geqslant (\log \xi)^A}}{\frac{1}{d_1^{1-2\varepsilon}}}{\rm .}$$
Comme la condition sur $d_1$ implique que l'une des valuations p-adiques de $d_1$ est plus grande que $\frac{5\log_2(\xi)}{2\log p_1}$, la contribution à la somme \eqref{eq Q1} des entiers $d_1\geqslant  (\log \xi)^B$ est 
$$\ll_q ||F||^{\varepsilon}\sigma^2\xi^2 \log(\xi)^{-0,4}{\rm .}$$
\end{proof}

Il nous reste à déterminer la vitesse de convergence de la somme sur $d_1$ dans la formule \eqref{eq Q1}. Pour cela, nous travaillons avec la quantité suivante
\begin{equation}
\label{eq Q1'}
Q_3(F,\xi,\mathcal{R}):=\sum\limits_{d_1 \mid q^{\infty}}{\sum\limits_{\substack{d_2 \mid q^{\infty}\\ d_1\leqslant (\log \xi)^A}}{\sum\limits_{\alpha \in G_1}{\sum\limits_{\beta \in W_{\alpha,d_1,d_2}}{Q_2(F_{\beta,d_1,d_2},\xi,\mathcal{R}_{\beta,d_1,d_2})}}}}
\end{equation} 
où 
$$A:= \frac{5\log p_{\omega(q)}}{2\log p_1}{\rm .}$$

\begin{lemme}
\label{conv1}
La contribution des entiers $d_1\geqslant \log(\xi)^{5}$ à la somme \eqref{eq Q1'} peut être incluse dans le terme d'erreur de \eqref{1}.
\end{lemme}
\begin{proof}
Lorsque $\xi/d_1\leqslant \xi^{\varepsilon}$, une majoration triviale fournit
$$\sum\limits_{\bfx \in \mathbb{Z}^2\cap \mathcal{R}_{\beta,d_1,d_2}(\xi)}{\tau_3(F_{\beta,d_1,d_2}(\bfx))} \ll \frac{\sigma^2\xi^{3/2+\varepsilon}}{d_1^{1/2}} {\rm .}$$
Nous verrons qu'il s'agit d'un terme d'erreur acceptable pour \eqref{1}. Dans le cas contraire, nous appliquons de nouveau le théorème 4 de \cite{He} pour évaluer cette somme. Nous obtenons
$$\sum\limits_{\bfx \in \mathbb{Z}^2\cap \mathcal{R}_{\beta,d_1,d_2}(\xi)}{\tau_3(F_{\beta,d_1,d_2}(\bfx))} \ll \frac{||F||^{\varepsilon}\sigma^2\xi^2(\log \xi)^2}{(d_1d_2)^{1-\varepsilon}} {\rm .}$$
Nous minorons  $d_1^{1-\varepsilon}$ par $d_1^{1/2}(\log \xi)^{B}$ où $B:=5/2-5\varepsilon$. En reportant dans la formule \eqref{eq Q1'}, nous obtenons que la contribution recherchée est 
\begin{align*}
\ll& \sum\limits_{\substack{d_1 \mid q^{\infty} \\ d_1\geqslant (\log \xi)^{2B}}}{\sum\limits_{\substack{d_2 \mid q^{\infty}\\ d_2\leqslant (\log \xi)^A}}{\sum\limits_{\alpha \in G_1}{\sum\limits_{\beta \in W_{\alpha,d_1,d_2}}{\frac{||F||^{\varepsilon}\sigma^2\xi^2(\log \xi)^2}{(\log \xi)^{B}d_1^{1/2}d_2^{1-\varepsilon}}}}}}\\
&+\sum\limits_{\substack{d_1 \mid q^{\infty} \\ d_1\geqslant (\log \xi)^{2B}}}{\sum\limits_{\substack{d_2 \mid q^{\infty}\\ d_2\leqslant (\log \xi)^A}}{\sum\limits_{\alpha \in G_1}{\sum\limits_{\beta \in W_{\alpha,d_1,d_2}}{\frac{||F||^{\varepsilon}\sigma^2\xi^{3/2+\varepsilon}}{d_1^{1/2}}}}}}{\rm .}
\end{align*}
Nous utilisons le Lemme \ref{triv} avec $\alpha=1/2$ afin d'obtenir, uniformément pour $d_2 \mid q^{\infty}$ et pour tout $d_3 \mid d_2q$,
$$\sum\limits_{\substack{d_1 \mid q^{\infty} \\ (d_1,d_2q)=d_3}}{\frac{1}{d_1^{1/2}}}\ll_q \frac{1}{d_3^{1/2}}{\rm .}$$
La contribution recherchée est donc
\begin{align*}
\ll_q & \frac{||F||^{\varepsilon}\sigma^2\xi^2(\log \xi)^2}{(\log \xi)^{B}}\sum\limits_{\substack{d_2 \mid q^{\infty}\\ d_2\leqslant (\log \xi)^A}}{\sum\limits_{d_3 \mid d_2q}{\sum\limits_{\alpha \in G_1}{\frac{|W_{\alpha,d_3,d_2}|}{d_3^{1/2}d_2^{1-\varepsilon}}}}}\\
&+||F||^{\varepsilon}\sigma^2\xi^{3/2+\varepsilon}\sum\limits_{\substack{d_2 \mid q^{\infty}\\ d_2\leqslant (\log \xi)^A}}{\sum\limits_{d_3 \mid d_2q}{\sum\limits_{\alpha \in G_1}{\frac{|W_{\alpha,d_3,d_2}|}{d_3^{1/2}}}}}\\
\ll_q&\frac{||F||^{\varepsilon}\sigma^2\xi^2(\log \xi)^2}{(\log \xi)^{B}}\sum\limits_{\substack{d_2 \mid q^{\infty}\\ d_2\leqslant (\log \xi)^A}}{\sum\limits_{d_3 \mid d_2q}{\sum\limits_{\alpha \in G_1}{\frac{|W_{\alpha,d_3,d_2}|}{d_3d_2^{1/2-\varepsilon}}}}}\\
&+||F||^{\varepsilon}\sigma^2\xi^{3/2+\varepsilon}\sum\limits_{\substack{d_2 \mid q^{\infty}\\ d_2\leqslant (\log \xi)^A}}{\sum\limits_{d_3 \mid d_2q}{\sum\limits_{\alpha \in G_1}{\frac{d_2^{1/2}|W_{\alpha,d_3,d_2}|}{d_3}}}}\\
\ll_q & \frac{||F||^{\varepsilon}\sigma^2\xi^2(\log \xi)^2}{(\log \xi)^{B}}+||F||^{\varepsilon}\sigma^2\xi^{3/2+\varepsilon}(\log \xi)^{3A/2}{\rm ,}
\end{align*}
ce qui est un terme d'erreur acceptable pour \eqref{1}.
\end{proof}
La proposition suivante nous permet d'estimer $Q_2(F_{\beta,d_1,d_2},\xi,\mathcal{R}_{\beta,d_1,d_2})$ lorsque $d_1$ et $d_2$ sont de tailles contrôlées.
\begin{prop}
\label{prop}
Soient $\varepsilon>0$, $\sigma>0$, $\xi>0$, $\kappa>0$, $\mathcal{R}$ un domaine de $\mathbb{R}^2$,\\ $J$ une forme binaire de degré $3$ irréductible sur $\mathbb{K}$, $F(\bfx)=J(M\bfx)/D \in~\mathbb{Z}[X,Y]$ avec $M \in M_2(\mathbb{Z})$ de déterminant non nul, $D \in \mathbb{N}^*$ et tel que $\chi(F(\bfx))=1$ pour tout $\bfx \in \mathcal{D}_q$ (défini en \eqref{Dq}). Lorsque les hypothèses (H1), (H2), (H3) et (H4) sont vérifiées et sous les conditions
$$\xi\geqslant ||M||^{\kappa}+\e^{\kappa D} \;\;\;\;\;1/\sqrt{\xi}\leqslant \sigma \leqslant \xi^{3/2}{\rm ,}\;\;\;\;\; 1/\sqrt{\xi}\leqslant \vartheta \leqslant \xi^{3/2}{\rm ,}$$
nous avons
\begin{equation}
\label{41}
Q_{2}(F,\xi,\mathcal{R})=K_1(F)L(1,\chi)L(1,\chi^2){\rm vol}(\mathcal{R})\xi^2+O\Big(\frac{||F||^{\varepsilon}(\sigma^2+\vartheta^2)\xi^2}{(\log \xi)^{\eta}}\Big)
\end{equation}
où $\eta$ est défini en \eqref{eta} et
$$K_1(F):=3\frac{\varphi(q)}{q}\prod\limits_{p\nmid q}{K_p(F)}{\rm .}$$
Les termes $K_p(F)$ sont définis en \eqref{K_p}.
\end{prop} 
\section{Démonstration de la Proposition \ref{prop} et du Théorème \ref{theo 1}}

\begin{proof}
Nous réécrivons dans un premier temps la formule \eqref{41} sous la forme suivante
\begin{equation}
\label{42}
Q_{2}(F,\xi,\mathcal{R})= \sum\limits_{\substack{1\leqslant d_1  \\ 1\leqslant d_2}}{\chi(d_1)\chi(d_2)\big|\Lambda(d_1d_2,F)\cap \mathcal{D}_q\cap\mathcal{R}(\xi)\big|} {\rm .}
\end{equation}

Pour estimer cette somme, nous mettons à profit la symétrie des couples $(d_1,d_2)$ tels que $d_1d_2\mid F(\bfx)$ autour des valeurs  $d_1=\vartheta \xi$ et $d_2=\vartheta \xi$ afin d'utiliser les Lemmes \ref{lemme 1} et \ref{lemme 2}. Cependant, la partition est plus difficile que celle réalisée dans le lemme 7.1 de \cite{B}. Pour la suite, nous posons 
\begin{align}
\label{z}
\begin{split}
z_1&:=\vartheta\xi(\log \xi)^{-4}{\rm ,} \\
z_2&:=\vartheta\xi(\log \xi)^{-2\delta}{\rm ,} \\
z_3&:=\vartheta\xi(\log \xi)^{\delta}{\rm ,} \\
z_4&:=\vartheta\xi(\log \xi)^{3} \\
\end{split}
\end{align}
et nous procédons à la partition suivante.

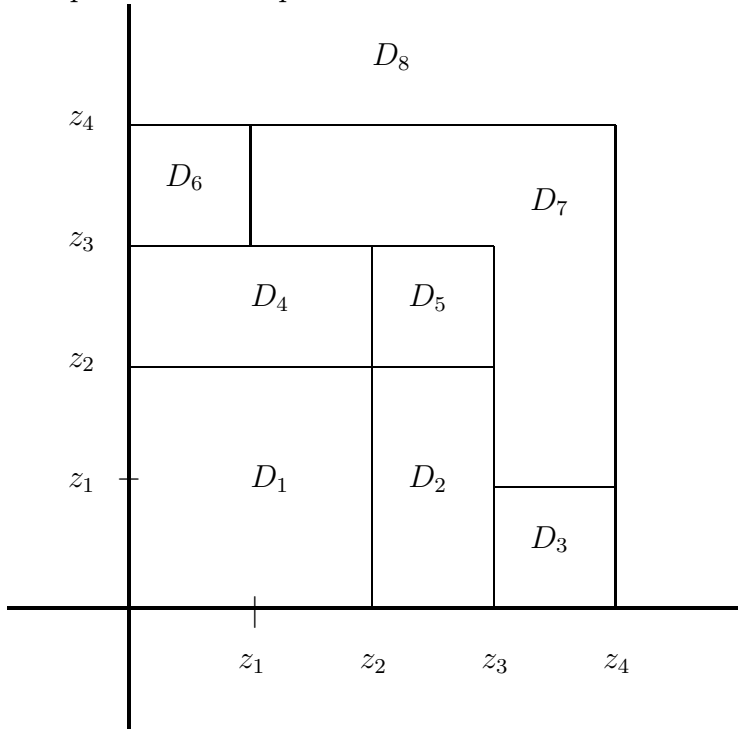
\setlength{\unitlength}{1.6cm}
\begin{picture}(6,6)
\linethickness{0.3mm}
\put(1,0){\line(0,1){6}}
\put(0,1){\line(1,0){6}}
\linethickness{0.15mm}
\multiput(3,1)(1,0){2}
{\line(0,1){3}}
\multiput(1,3)(0,1){2}
{\line(1,0){3}}
\put(5,1){\line(0,1){4}}
\put(1,5){\line(1,0){4}}
\put(2,4){\line(0,1){1}}
\put(4,2){\line(1,0){1}}
\put(2,0.9){$|$}
\put(0.9,2){$-$}
\put(1.9,0.5){$z_1$}
\put(2.9,0.5){$z_2$}
\put(3.9,0.5){$z_3$}
\put(4.9,0.5){$z_4$}
\put(0.5,2){$z_1$}
\put(0.5,3){$z_2$}
\put(0.5,4){$z_3$}
\put(0.5,5){$z_4$}
\put(2,2){$D_1$}
\put(3.3,2){$D_2$}
\put(4.3,1.5){$D_3$}
\put(2,3.5){$D_4$}
\put(3.3,3.5){$D_5$}
\put(1.3,4.5){$D_6$}
\put(4.3,4.3){$D_7$}
\put(3,5.5){$D_8$}
\end{picture}
\\ 

Nous notons $C_j$ la contribution à $Q_2$ des couples $(d_1,d_2)$ appartenant au domaine $D_j$.
Nous verrons que seuls les termes $C_1$ et $C_8$ contribuent au terme principal. Par ailleurs, nous verrons que la contribution $C_5$ ne relève pas du Lemme \ref{lemme 2} et de ses corollaires et nécessite un traitement particulier. \\ 

Pour estimer $C_1$, nous utilisons le Lemme \ref{lemme 2} avec $y_1=y_2=z_2$, nous obtenons ainsi
\begin{align*}
C_1=&\frac{\varphi(q)}{q}{\rm vol}(\mathcal{R})\xi^2\sum\limits_{\substack{d_1\leqslant z_2 \\ d_2\leqslant z_2}}{\chi(d_1)\chi^2(d_2)\frac{\varrho_F^+(d_1d_2)}{d_1^2d_2^2}}\\
&+O\Big(\frac{||F||^{\varepsilon}(\sigma^2+\vartheta^2)\xi^2\mathcal{L}(\sigma\xi)^{\sqrt{3}+\varepsilon}}{(\log\xi)^{2\delta}}\Big)
\end{align*}
L'estimation \eqref{eq coro 2} fournit alors 
\begin{align}
\label{zone 1}
\begin{split}
C_1=&\frac{\varphi(q)}{q}{\rm vol}(\mathcal{R})\xi^2\prod\limits_{p\nmid q}{\Big(\sum\limits_{\nu\geqslant 0}{\frac{\varrho_F^+(p^{\nu})}{p^{2\nu}}(\chi*\chi^2)(p^{\nu})}\Big)} \\
&+O\Big(\frac{||F||^{\varepsilon}(\sigma^2+\vartheta^2)\xi^2\mathcal{L}(\sigma\xi)^{\sqrt{3}+\varepsilon}}{(\log\xi)^{4\delta}}\Big){\rm .}
\end{split}
\end{align}

Pour estimer $C_8$, nous utilisons un changement de variable et le Corollaire~\ref{lemme 1}.

\begin{align*}
C_8=&\sum\limits_{\substack{z_4< d_1 \\ 1\leqslant d_2}}{\chi(d_1)\chi^2(d_2)\big|\Lambda(d_1d_2,F)\cap \mathcal{D}_q\cap\mathcal{R}(\xi)\big|}\\
&+\sum\limits_{\substack{1\leqslant d_1\\ z_4< d_2 }}{\chi(d_1)\chi^2(d_2)\big|\Lambda(d_1d_2,F)\cap \mathcal{D}_q\cap\mathcal{R}(\xi)\big|} \\
&-\sum\limits_{\substack{z_4< d_1\\ z_4< d_2 }}{\chi(d_1)\chi^2(d_2)\big|\Lambda(d_1d_2,F)\cap\mathcal{D}_q\cap\mathcal{R}(\xi)\big|}\\
\end{align*}
Nous exploitons alors l'égalité $\chi(F(\bfx))=1$ de la manière suivante : pour la première et la troisième somme, nous posons $d_1'=\frac{F(\bfx)}{d_1d_2}$, pour la deuxième, nous posons $d_2'=\frac{F(\bfx)}{d_1d_2}$. En posant
\begin{equation}
\label{Ri}
\mathcal{R}_i(\xi,d):=\{\bfx \in \mathcal{R}(\xi) \mbox{ : } d\leqslant F(\bfx)/z_i\}{\rm ,}
\end{equation}
nous obtenons
\begin{align*}
C_8=&\sum\limits_{d_1'd_2< \vartheta^3\xi^3/z_4}{\chi^2(d_1')\chi(d_2)\big|\Lambda(d_1'd_2,F)\cap \mathcal{D}_q\cap\mathcal{R}_4(\xi,d_1'd_2)\big|}\\
&+\sum\limits_{d_1d_2'< \vartheta^3\xi^3/z_4}{\chi^2(d_1)\chi(d_2')\big|\Lambda(d_1d_2',F)\cap \mathcal{D}_q\cap\mathcal{R}_4(\xi,d_1d_2')\big|} \\
&-\sum\limits_{\substack{d_1'd_2< \vartheta^3\xi^3/z_4\\ z_4< d_2 }}{\chi^2(d_1')\chi(d_2)\big|\Lambda(d_1'd_2,F)\cap \mathcal{D}_q\cap\mathcal{R}_4(\xi,d_1'd_2)\big|}{\rm .}
\end{align*}
Nous utilisons le Corollaire \ref{lemme 1} pour estimer ces différentes sommes en posant $y=\vartheta^3\xi^3/z_4$. 
\begin{align*}
C_8=&\frac{\varphi(q)}{q}\Big(\sum\limits_{d_1'd_2< \vartheta^3\xi^3/z_4}{{\rm vol}\big(\mathcal{R}_4(\xi,d_1'd_2)\big)\chi^2(d_1')\chi(d_2)\frac{\varrho_F^+(d_1'd_2)}{d_1'^2d_2^2}}\\
&+\sum\limits_{d_1d_2'< \vartheta^3\xi^3/z_4}{{\rm vol}\big(\mathcal{R}_4(\xi,d_1d_2')\big)\chi^2(d_1)\chi(d_2')\frac{\varrho_F^+(d_1d_2')}{d_1^2d_2'^2}} \\
&-\sum\limits_{\substack{d_1'd_2< \vartheta^3\xi^3/z_4\\ z_4\leqslant d_2\leqslant \vartheta^3\xi^3/z_4 }}{{\rm vol}\big(\mathcal{R}_4(\xi,d_1'd_2)\big)\chi^2(d_1')\chi(d_2)\frac{\varrho_F^+(d_1'd_2)}{d_1'^2d_2^2}}\Big)\\
&+O\Big(\frac{||F||^{\varepsilon}(\sigma^2+\vartheta^2)\xi^2}{(\log\xi)^{1/3}}\Big){\rm .}
\end{align*}
Notons que 
\begin{equation}
\label{RRi}
\mathcal{R}(\xi)\backslash\mathcal{R}_i(\xi,d)=\{\bfx \in \mathcal{R}(\xi) \mbox{ : }  F(\bfx)/z_i<d\}{\rm .}
\end{equation}
La majoration \eqref{eq coro 1} et une adaptation de l'estimation \eqref{rhod1} fournissent alors
\begin{align*}
C_8=&2\frac{\varphi(q)}{q}{\rm vol}(\mathcal{R})\xi^2\prod\limits_{p\nmid q}{\Big(\sum\limits_{\nu\geqslant 0}{\frac{\varrho_F^+(p^{\nu})}{p^{2\nu}}(\chi*\chi^2)(p^{\nu})}\Big)}\\
&+O\Big(\int_{\mathcal{R}(\xi)}{\frac{\log \xi}{\log^2(F(\bfx)/z_4^2+2)}{\rm d}\bfx}+\frac{||F||^{\varepsilon}(\sigma^2+\vartheta^2)\xi^2}{(\log\xi)^{1/3}}\Big){\rm .}
\end{align*}
Par ailleurs, nous avons
\begin{equation}
\label{majint}
\int_{\mathcal{R}(\xi)}{\frac{\log \xi}{\log^2(F(\bfx)/z_4^2+2)}{\rm d}\bfx}\ll \frac{||F||^{\varepsilon}\sigma^2\xi^2}{\log (\xi) \mbox{d\'et}(M)}{\rm ,}
\end{equation}
ainsi,
\begin{align}
\label{zone 8}
\begin{split}
C_8=&2\frac{\varphi(q)}{q}{\rm vol}(\mathcal{R})\xi^2\prod\limits_{p\nmid q}{\Big(\sum\limits_{\nu\geqslant 0}{\frac{\varrho_F^+(p^{\nu})}{p^{2\nu}}(\chi*\chi^2)(p^{\nu})}\Big)}\\
&+O\Big(\frac{||F||^{\varepsilon}(\sigma^2+\vartheta^2)\xi^2}{(\log\xi)^{1/3}}\Big){\rm .}
\end{split}
\end{align}
Pour estimer $C_2$, nous faisons appel au Lemme \ref{lemme 2}, avec $y_1=z_3$ et $y_2=z_2$ et à la majoration \eqref{eq coro 2} afin obtenir
\begin{align}
\label{zone 2}
C_2=&\frac{\varphi(q)}{q}{\rm vol}(\mathcal{R})\xi^2\sum\limits_{\substack{z_2< d_1\leqslant z_3 \\ d_2\leqslant z_2}}{\chi(d_1)\chi^2(d_2)\frac{\varrho_F^+(d_1d_2)}{d_1^2d_2^2}}\nonumber\\
&+ O\Big(\frac{||F||^{\varepsilon}(\sigma^2+\vartheta^2)\xi^2\mathcal{L}(\sigma \xi)^{\sqrt{3}+\varepsilon}}{(\log\xi)^{\delta}}\Big)\nonumber\\
=&O\Big(\frac{||F||^{\varepsilon}(\sigma^2+\vartheta^2)\xi^2\mathcal{L}(\sigma \xi)^{\sqrt{3}+\varepsilon}}{(\log\xi)^{\delta/2}}\Big){\rm .}
\end{align}

La contribution $C_4$ s'estime de manière analogue, en échangeant simplement les valeurs de $y_1$ et $y_2$. Ainsi,
\begin{equation}
\label{zone 4}
C_4=O\Big(\frac{||F||^{\varepsilon}(\sigma^2+\vartheta^2)\xi^2\mathcal{L}(\sigma \xi)^{\sqrt{3}+\varepsilon}}{(\log\xi)^{\delta/2}}\Big){\rm .}
\end{equation}

 Les contributions $C_3$ et $C_6$ s'estiment également de la même manière, en posant simplement $y_1=z_4$ et $y_2=z_1$ pour $C_3$, et l'inverse pour $C_6$. Nous obtenons
  \begin{equation}
\label{zone 3+6}
C_3+C_6=O\Big(\frac{||F||^{\varepsilon}(\sigma^2+\vartheta^2)\xi^2\mathcal{L}(\sigma \xi)^{\sqrt{3}+\varepsilon}}{\sqrt{\log \xi}}\Big){\rm .}
\end{equation}

Pour estimer la contribution $C_7$, nous écrivons
\begin{align*}
C_7=&\sum\limits_{\bfx\in \mathcal{D}_q \cap\mathcal{R}(\xi)}{\sum\limits_{\substack{d_1d_2\mid F(\bfx) \\ z_1< d_1\leqslant z_4 \\ z_3< d_2\leqslant z_4}}{\chi(d_1)\chi^2(d_2)}}\\
&+\sum\limits_{\bfx\in \mathcal{D}_q\cap \mathcal{R}(\xi)}{\sum\limits_{\substack{d_1d_2\mid F(\bfx) \\ z_3< d_1\leqslant z_4 \\ z_1< d_2\leqslant z_4}}{\chi(d_1)\chi^2(d_2)}}\\
&-\sum\limits_{\bfx\in \mathcal{D}_q\cap \mathcal{R}(\xi)}{\sum\limits_{\substack{d_1d_2\mid F(\bfx) \\ z_3< d_1\leqslant z_4 \\ z_3< d_2\leqslant z_4}}{\chi(d_1)\chi^2(d_2)}}\\
:=&C_{71}+C_{72}-C_{73}{\rm ,}
\end{align*}
où $z_1$, $z_3$ et $z_4$ sont définis en \eqref{z}. Les trois sommes $C_{71}$, $C_{72}$ et $C_{73}$ se traitant de la même manière, nous ne nous intéressons qu'à $C_{71}$. Nous effectuons alors le changement de variable $d_2'=F(\bfx)/(d_1d_2)$ afin d'obtenir

\begin{align*}
C_{71}=&\sum\limits_{\bfx\in \mathcal{D}_q\cap \mathcal{R}(\xi)}{\sum\limits_{\substack{d_1d_2'\mid F(\bfx) \\ z_1< d_1\leqslant z_4 \\ F(\bfx)/z_4\leqslant d_1d_2'< F(\bfx)/z_3}}{\!\!\!\!\!\!\!\!\!\!\!\!\!\chi^2(d_1)\chi(d_2')}}\\
=&\sum\limits_{\substack{z_1< d_1\leqslant z_4 \\  d_1d_2'< \vartheta^3\xi^3/z_3}}{\!\!\!\!\!\!\chi^2(d_1)\chi(d_2')\big|\Lambda(d_1d_2',F)\cap \mathcal{D}_q\cap \big(\mathcal{R}_3(\xi,d_1d_2')\backslash \mathcal{R}_4(\xi,d_1d_2')\big)\big|} \\
=&\sum\limits_{k\in I}{\!\!\sum\limits_{\substack{\e^k< d_1\leqslant \e^{k+1} \\  d_1d_2'< \vartheta^3\xi^3/z_3}}{\!\!\!\!\!\chi^2(d_1)\chi(d_2')\big|\Lambda(d_1d_2',F)\cap \mathcal{D}_q\cap\big(\mathcal{R}_3(\xi,d_1d_2')\backslash \mathcal{R}_4(\xi,d_1d_2')\big)\big|}}{\rm .}
\end{align*}
où $\mathcal{R}_i(\xi,d)$ est défini en \eqref{Ri} et $I:=\mathbb{Z}\cap [\log\xi-4\log_2\xi,\log\xi+3\log_2\xi]$. 
La somme extérieure comportant $O(\log_2\xi)$ termes, la majoration \eqref{maj coro} et l'égalité \eqref{RRi} fournissent
\begin{align*}
C_{71}\ll& \int_{\mathcal{R}(\xi)}{\sum\limits_{\substack{z_1< d_1\leqslant z_4 \\ F(\bfx)/z_4\leqslant d_1d_2'}}{\!\!\!\chi^2(d_1)\chi(d_2')\frac{\varrho_F^+(d_1d_2')}{d_1^2d_2'^2}}{\rm d}\bfx}\\
&+\frac{||F||^{\varepsilon}(\sigma^2+\vartheta^2)\xi^2\mathcal{L}(\sigma \xi)^{\sqrt{3}+\varepsilon}\log_2\xi}{(\log\xi)^{\delta/2}}{\rm .}
\end{align*}
Une adaptation de l'estimation \eqref{rhod1} fournit alors
\begin{align*}
C_{71}\ll&\int_{\mathcal{R}(\xi)}{\frac{\log_2 \xi}{\log^2(F(\bfx)/z_4^2+2)}{\rm d}\bfx}\\
&+\frac{||F||^{\varepsilon}(\sigma^2+\vartheta^2)\xi^2\mathcal{L}(\sigma \xi)^{\sqrt{3}+\varepsilon}\log_2\xi}{(\log\xi)^{\delta/2}}{\rm .}
\end{align*}
L'intégrale étant majorée par \eqref{majint}, nous obtenons
\begin{equation}
\label{zone 7}
C_{71}\ll \frac{||F||^{\varepsilon}(\sigma^2+\vartheta^2)\xi^2\mathcal{L}(\sigma \xi)^{\sqrt{3}+\varepsilon}\log_2\xi}{(\log\xi)^{\delta/2}}{\rm .}
\end{equation}

 Il nous reste à majorer $C_5$. Posons
 $$\Upsilon(\Xi,\xi,\vartheta,b):=\sum\limits_{\substack{\bfx \in  \mathcal{R}(\Xi) \\ (m,n)=1}}{\sum\limits_{\substack{d_1d_2\mid b^3F(\bfx) \\ (d_1,d_2) \in D_5}}{\chi(d_1)\chi^2(d_2)}}{\rm ,}$$
 de sorte que
 $$C_5=\sum\limits_{(b,q)=1}{\Upsilon(\xi/b,\xi,\vartheta,b)}{\rm .}$$
 Notons la majoration triviale
 $$\Upsilon(\Xi,\xi,\vartheta,b)\leqslant \tau_3(b^3)\sum\limits_{\bfx\in \mathbb{Z}^2\cap \mathcal{R}(\Xi)}{\tau_3(F(\bfx))}\ll_{\varepsilon}\tau_3(b^3)||F||^{\varepsilon}\sigma^2\Xi^2(\log \Xi)^2$$
 qui résulte des Lemmes \ref{theo B} et \ref{lemme B}. Cette majoration fournit l'estimation suivante
 \begin{equation}
 \label{Upsilon}
 C_5=\sum\limits_{\substack{b\leqslant L \\(b,q)=1}}{\Upsilon(\xi/b,\xi,\vartheta,b)}+O_{\varepsilon}\Big(\frac{||F||^{\varepsilon}\sigma^2\xi^2}{\sqrt{\log \xi}}\Big) {\rm ,}
 \end{equation}
 où nous avons posé $L:=(\log \xi)^5$. Déterminons une majoration de $\Upsilon(\Xi,\xi,\vartheta,b)$ lorsque $\xi/L\leqslant \Xi\leqslant \xi$ et $b\leqslant L$. 
\begin{equation}
\Upsilon(\Xi,\xi,\vartheta,b)=\sum\limits_{\substack{\bfx\in \mathbb{Z}^2\cap\mathcal{R}(\Xi)\\(m,n)=1}}{\sum\limits_{\substack{d_1d_2\mid b^3F(\bfx)\\ z_2< d_1\leqslant z_3 \\ z_2< d_2\leqslant z_3 }}{\chi(d_1)\chi^2(d_2)}}{\rm .}
\end{equation}
Nous séparons tout d'abord la somme selon la valeur de $\Omega_z(F(\bfx))$, où 
$$\Omega_z(n):=\sum\limits_{\substack{p>z\\ p^{\nu}||n}}{\nu}$$
et où $z=z(\varepsilon)$ vérifie $\log 2/\log z>\varepsilon$. Ainsi, toute fonction sous multiplicative $f$ vérifiant $f\leqslant 2^{\Omega_z}$ vérifie les hypothèses du Lemme \ref{theo B} pour la forme $F$. Nous obtenons donc
\begin{align}
\label{44}
\begin{split}
\Upsilon(\Xi,\xi,\vartheta,b)= & \sum\limits_{\substack{\bfx\in  \mathbb{Z}^2\cap\mathcal{R}(\Xi)\\ (m,n)=1 \\\Omega_z(F(\bfx))\leqslant 3a\log_2\xi}}{\sum\limits_{\substack{d_1d_2\mid b^3F(\bfx)\\ z_2< d_1\leqslant z_3 \\ z_2< d_2\leqslant z_3 }}{\chi(d_1)\chi^2(d_2)}}\\
&+\sum\limits_{\substack{\bfx\in \mathbb{Z}^2\cap\mathcal{R}(\Xi)\\ (m,n)=1 \\ \Omega_z(F(\bfx))> 3a\log_2\xi}}{\sum\limits_{\substack{d_1d_2\mid b^3F(\bfx)\\ z_2< d_1\leqslant z_3 \\ z_2< d_2\leqslant z_3 }}{\chi(d_1)\chi^2(d_2)}}
\end{split}
\end{align}
où $a$ est un paramètre que nous déterminerons par la suite. En observant que la somme intérieure est $\ll \tau_3(b^3)\Delta_3(F(\bfx),\chi)(\log_2\xi)^{2}$, nous appliquons l'inégalité de Cauchy Schwarz à chacun des deux membres de droite de \eqref{44}. Il vient
\begin{align}
\label{45}
\begin{split}
\Upsilon(\Xi,\xi,\vartheta,b)\ll & \tau_3(b^3)(\log_2\xi)^2B_1(\Xi,\xi,\vartheta,b,\delta,a)^{1/2}\Big(\!\!\!\!\sum\limits_{\substack{\bfx\in\mathbb{Z}^2\cap \mathcal{R}(\Xi) \\\Omega_z(F(\bfx))\leqslant 3a\log_2\xi}}{\!\!\!\!\!\!\Delta_3^2(F(\bfx),\chi,\chi^2)}\Big)^{1/2} \\
&+ \tau_3(b^3)(\log_2\xi)^2B_2(\Xi,\xi,\vartheta,b,\delta,a)^{1/2}\Big(\!\!\!\!\!\!\!\!\sum\limits_{\substack{\bfx\in \mathbb{Z}^2\cap \mathcal{R}(\Xi)\\\Omega_z(F(\bfx))> 3a\log_2\xi}}{\!\!\!\!\!\!\!\!\Delta_3^2(F(\bfx),\chi,\chi^2)}\Big)^{1/2} \\
\end{split}
\end{align}
où nous avons posé 
$$B_1(\Xi,\xi,\vartheta,b,\delta,a):=\sum\limits_{\substack{\bfx\in \mathbb{Z}^2\cap\mathcal{R}(\Xi) \\ (m,n)=1 \\ \exists d_1d_2\mid b^3F(\bfx) \\ z_2< d_1,d_2\leqslant z_3 \\ \Omega_z(F(\bfx))\leqslant 3a\log_2\xi}}{1}$$
et
$$B_2(\Xi,\xi,\vartheta,b,\delta,a):=\sum\limits_{\substack{\bfx\in \mathbb{Z}^2\cap\mathcal{R}(\Xi) \\(m,n)=1 \\ \exists d_1d_2\mid b^3F(\bfx) \\ z_2< d_1,d_2\leqslant z_3 \\ \Omega_z(F(\bfx))> 3a\log_2\xi}}{1}{\rm .}$$

Nous commençons par majorer la contribution du second terme du membre de droite de \eqref{45}. Pour cela, nous devons estimer $B_2(\Xi,\xi,\vartheta,b,\delta,a)$. La méthode paramétrique fournit, pour tout $y\geqslant 1$,
\begin{align*}
B_2(\Xi,\xi,\vartheta,b,\delta,a)& \leqslant \sum\limits_{\substack{\bfx\in \mathbb{Z}^2\cap\mathcal{R}(\Xi)\\ (m,n)=1 \\ \exists d_1d_2\mid b^3F(\bfx) \\ z_2< d_1,d_2\leqslant z_3 }}{y^{\Omega_z(F(\bfx))-3a\log_2\xi}} \\
&\leqslant \sum\limits_{\substack{\bfx \in \mathbb{Z}^2 \\||\bfx||\leqslant \sigma \Xi}}{y^{\Omega_z(F(\bfx))-3a\log_2\xi}}{\rm .}
\end{align*}
En choisissant $y=3a$, ce qui n'est légitime que si $3a\geqslant 1$, nous obtenons, en vertu du Lemme \ref{theo B}
\begin{equation}
\label{48}
B_2(\Xi,\xi,\vartheta,b,\delta,a)\ll \frac{||F||^{\varepsilon}\sigma^2 \Xi^2}{(\log\xi)^{Q(3a)}}
\end{equation}
où la fonction $Q$ est définie par 
\begin{align*}
Q&:&&\mathbb{R}_+^* &&\longrightarrow &&\mathbb{R} \\
& && x &&\longrightarrow && \int_1^x{\log (t) {\rm d}t} {\rm .}
\end{align*}
Ainsi
$$Q(x)=x\log x-x+1{\rm .}$$
Nous estimons à présent
$$\sum\limits_{\substack{\bfx\in \mathbb{Z}^2\cap\mathcal{R}(\Xi) \\ \Omega_z(F(\bfx))> 3a\log_2\xi}}{\Delta_3^2(F(\bfx),\chi,\chi^2)}{\rm .}$$
Pour cela, nous utilisons les Lemmes \ref{theo B}, \ref{lemme B}, ainsi que le Théorème \ref{theo L}  afin d'obtenir
\begin{align}
\label{49}
\begin{split}
\sum\limits_{\substack{\bfx\in \mathbb{Z}^2\cap\mathcal{R}(\Xi) \\ \Omega_z(F(\bfx))> 3a\log_2\xi}}{\Delta_3^2(F(\bfx),\chi,\chi^2)}&\leqslant  \sum\limits_{\bfx\in \mathbb{Z}^2\cap\mathcal{R}(\Xi)}{\Delta_3^2(F(\bfx),\chi,\chi^2)} \\
&\ll ||F||^{\varepsilon} \sigma^2 \Xi^2(\log\xi)^{\rho}\mathcal{L}(\xi)^{\alpha}
\end{split}
\end{align}
En combinant \eqref{48} et \eqref{49}, nous obtenons que le second terme du membre de droite de \eqref{45} est 
\begin{equation}
\label{50}
\ll \tau_3(b^3)(\log_2\xi)^2||F||^{\varepsilon}\sigma^2\Xi^2 (\log\xi)^{\rho-Q(3a)}\mathcal{L}(\xi)^{\alpha/2} {\rm .}
\end{equation}

Nous estimons à présent la contribution du premier terme du membre de droite de \eqref{45}. Majorons dans un premier temps
$$\sum\limits_{\substack{\bfx\in \mathbb{Z}^2\cap\mathcal{R}(\Xi) \\\Omega_z(F(\bfx))\leqslant 3a\log_2\xi}}{\Delta_3^2(F(\bfx),\chi,\chi^2)} {\rm .}$$
Pour cela, nous utilisons à nouveau la méthode paramétrique.
$$\sum\limits_{\substack{\bfx\in \mathbb{Z}^2\cap\mathcal{R}(\Xi) \\\Omega_z(F(\bfx))\leqslant 3a\log_2\xi}}{\Delta_3^2(F(\bfx),\chi,\chi^2)}  \leqslant \sum\limits_{\bfx\in \mathbb{Z}^2\cap\mathcal{R}(\Xi)}{y^{\Omega_z(F(\bfx))-3a\log_2\xi}\Delta_3^2(F(\bfx),\chi,\chi^2)}$$
valable pour tout $y\leqslant 1$. En appliquant les Lemmes \ref{theo B}, \ref{lemme B} et le Théorème \ref{theo L} avec $y=1/(1+\rho)$, nous obtenons
\begin{equation}
\label{51}
\sum\limits_{\substack{\bfx\in \mathbb{Z}^2\cap\mathcal{R}(\Xi) \\\Omega_z(F(\bfx)) \leqslant 3a\log_2\xi}}\ll ||F||^{\varepsilon}\sigma^2\Xi^2 (\log\xi)^{3a\log(\rho+1)-\rho/(\rho+1)}\mathcal{L}(\xi)^{\alpha} {\rm .}
\end{equation} 

Il ne nous reste plus qu'à estimer $B_1(\Xi,\xi,\vartheta,b,\delta,a)$. Nous séparons pour cela la somme en trois. Nous désignons ainsi par $B_{1j} (1\leqslant j \leqslant 3)$ les contributions à $B_1(\xi,\delta, a)$ des vecteurs $\bfx$ et des diviseurs $d_1$, $d_2$ satisfaisant respectivement les conditions
\begin{align*}
(B_{11})& && \Omega(d_1d_2) \leqslant 2a\log_2\xi \\
(B_{12})& && F(\bfx)\leqslant \vartheta^3\Xi^3/(\log \xi)^{3/(2\kappa)} \\
(B_{13}) & &&\Omega(d_1d_2) > 2a\log_2\xi{\rm ,} && F(\bfx)>\vartheta^3\Xi^3/(\log \xi)^{3/(2\kappa)}
\end{align*} 

Pour estimer $B_{11}$, observons dans un premier temps que l'existence d'un couple $(d_1,d_2)$ tel que $d_1d_2\mid b^3F(\bfx)$ avec $z_2< d_i\leqslant z_3$ implique l'existence d'un couple $(d_1',d_2')$ tel que $d_1'd_2'\mid F(\bfx)$ avec $z_2/b^3\leqslant d_i'\leqslant z_3$. Dans un second temps, nous utilisons la méthode paramétrique. Pour $y\leqslant 1$, nous avons, d'après le lemme 5.1 de \cite{B}, les Lemmes \ref{theo B} et \ref{lemme B}
\begin{align*}
B_{11}&\leqslant \sum\limits_{z_2/b^3< d_1,d_2\leqslant z_3}{y^{\Omega_z(d_1d_2)-2a\log_2\xi}\sum\limits_{\substack{\bfx \in \Lambda(d_1d_2,F)\cap \mathcal{R}(\Xi)  \\ (m,n)=1}}{1}} \\
&\ll (\log \xi)^{-2a \log y}\sum\limits_{z_2/b^3< d_1,d_2\leqslant z_3}{y^{\Omega_z(d_1d_2)}\frac{\rho_F^+(d_1d_2)}{d_1d_2}\Big(\frac{\sigma^2\Xi^2}{d_1d_2}+1\Big)} \\
&\ll ||F||^{\varepsilon}(\sigma^2\Xi^2+z_3^2)(\log_2\xi)^2(\log \xi)^{-2a\log y +2y-2}
\end{align*}
Pour le choix de $y=a$, qui n'est valable que pour $a\leqslant 1$, nous obtenons, en reportant dans la majoration ci-dessus la valeur de $z_3$
\begin{equation}
\label{52}
B_{11} \ll ||F||^{\varepsilon}(\sigma^2+\vartheta^2)\xi^2(\log_2\xi)^2(\log \xi)^{2\delta-2Q(a)}
\end{equation}
Le lemme 3.4 de \cite{D} fournit la majoration suivante.
$${\rm vol}\{\bfx \in \mathcal{R}\;:\; |F(\bfx)|/\vartheta^3\leqslant \beta\}\ll \vartheta^2\beta^{2/3}\sqrt{D}{\rm .} $$
En l'appliquant à $\beta=(\log \xi)^{-3/(2\kappa)}$,  nous obtenons

\begin{equation}
\label{53}
B_{12} \ll ||F||^{\varepsilon}\vartheta^2\Xi^2/\sqrt{\log \xi} {\rm .}
\end{equation}

Il nous reste à estimer $B_{13}$. Pour cela, nous utilisons le fait que puisque $\Omega_z(d_1d_2)>2a \log_2 \xi$, l'un des deux, que nous noterons $d_i$, vérifie $\Omega_z(d_i)>~~a \log_2 \xi$, notons $d_j$ l'autre diviseur et $d_k=b^3F(\bfx)/d_id_j$. Nous avons alors
$$\Omega_z(d_jd_k)=\Omega(b^3F(\bfx)/d_i)< \Omega_z(b^3)+2a\log_2 \xi {\rm .}$$
et
$$\vartheta^3\Xi^3/\{z_3^2(\log \xi)^{3/(2\kappa)+2\delta}\}< d_k< b^3\vartheta^3\Xi^3/z_2^2 {\rm .}$$
Nous faisons de nouveau appel au lemme 5.1 de \cite{B}, au Lemme \ref{theo B} et au Lemme \ref{lemme B} pour majorer $B_{13}$, nous obtenons
\begin{align*}
B_{13}&\leqslant &&\!\!\!\!\!\!\!\!\!\!\!\!\sum\limits_{\substack{z_2\leqslant d_1\leqslant z_3 \\ \vartheta^3\Xi^3/\{z_3^2(\log \xi)^{3/(2\kappa)}\}\leqslant d_2\leqslant b^3\vartheta^3\Xi^3/z_2^2 }}{y^{\Omega_z(d_1d_2)-2a\log_2\xi-\Omega_z(b^3)}\sum\limits_{\substack{\bfx \in \Lambda(d_1d_2,F)\cap \mathcal{R}(\Xi) \\ (m,n)=1}}{1}} \\
&\ll && y^{-\Omega_z(b^3)}(\log \xi)^{-2a \log y}\!\!\!\!\!\!\!\!\!\!\!\!\!\!\!\!\!\!\!\!\!\!\!\!\sum\limits_{\substack{z_2\leqslant d_1\leqslant z_3\\ \vartheta^3\Xi^3/\{z_3^2(\log \xi)^{3/(2\kappa)}\}\leqslant d_2\leqslant b^3\vartheta^3\Xi^3/z_2^2}}{\!\!\!\!\!\!\!\!\!\!\!\!y^{\Omega_z(d_1d_2)}\frac{\rho_F^+(d_1d_2)}{d_1d_2}\Big(\frac{\sigma^2\Xi^2}{d_1d_2}+1\Big)} \\
&\ll &&y^{-\Omega_z(b^3)}||F||^{\varepsilon}\Big(\sigma^2\Xi^2+\vartheta^3b^3\frac{\Xi^3z_3}{z_2^2}\Big)(\log_2\xi)^2(\log \xi)^{-2a\log y +2y-2}
\end{align*}
En choisissant $y=a$ et en reportant les valeurs de $z_2$ et $z_3$, nous obtenons
\begin{equation}
\label{54}
B_{13} \ll a^{-\Omega_z(b^3)}||F||^{\varepsilon}\Big(\sigma^2\Xi^2+\vartheta^2b^3\frac{\Xi^3}{\xi}\Big)(\log_2\xi)^2(\log \xi)^{5\delta-2Q(a)}{\rm .}
\end{equation}
Les équations \eqref{52}, \eqref{53} et \eqref{54} fournissent
\begin{equation}
\label{55}
B_1(\Xi,\xi,\vartheta,b,\delta,a)\ll a^{-\Omega_z(b^3)}||F||^{\varepsilon}\Big(\sigma^2\Xi^2+\vartheta^2b^3\frac{\Xi^3}{\xi}\Big)(\log_2 \xi)^2(\log \xi)^{5\delta-2Q(a)}{\rm .}
\end{equation}
Les équations \eqref{51} et \eqref{55} montrent alors que le premier terme du membre de droite de \eqref{45} est 
\begin{align}
\begin{split}
\label{56}
\ll& \tau_3(b^3)a^{-\Omega_z(b^3)/2}(\log_2 \xi)^{3}||F||^{\varepsilon}\Big(\sigma^2\Xi^2+\vartheta^2b^3\frac{\Xi^3}{\xi}\Big)^{1/2}\\
&\sigma\Xi(\log \xi)^{1/2(3a\log(\rho+1)-\rho/(\rho+1)-2Q(a)+5\delta)}\mathcal{L}(\xi)^{\alpha/2}{\rm .}
\end{split}
\end{align}
En reportant les équations \eqref{50} et \eqref{56} dans l'équation \eqref{45}, nous obtenons
\begin{align}
\label{57}
\begin{split}
\Upsilon(\Xi,\xi,\vartheta,b)\ll &\tau_3(b^3)a^{-\Omega_z(b^3)/2}||F||^{\varepsilon}\Big(\sigma^2\Xi^2+\vartheta^2b^3\frac{\Xi^3}{\xi}\Big)^{1/2}\\
&\sigma\Xi\mathcal{L}(\xi)^{\alpha/2}(\log_2 \xi)^3 \Big( (\log \xi)^{\gamma_1}+(\log \xi)^{\gamma_2}\Big){\rm .}
\end{split}
\end{align}
En reportant dans \eqref{Upsilon}, nous obtenons
\begin{align}
\begin{split}
&\sum\limits_{\bfx\in \mathcal{D}_q\cap \mathcal{R}(\xi)}{\sum\limits_{\substack{d_1d_2\mid F(\bfx)\\ (d_1,d_2) \in 5}}{\chi(d_1)\chi^2(d_2)}}\\
\ll &||F||^{\varepsilon}(\sigma^2+\vartheta^2)\xi^2\mathcal{L}(\xi)^{\alpha/2}(\log_2 \xi)^{O_a(1)} \Big( (\log \xi)^{\gamma_1}+(\log \xi)^{\gamma_2}\Big){\rm .}
\end{split}
\end{align}
où nous avons posé 
$$\gamma_1=\frac{1}{2}\Big(3a\log(\rho+1)-\frac{\rho}{\rho+1}-2Q(a)+5\delta\Big)$$ 
et $$\gamma_2=\frac{1}{2}\big(\rho-Q(3a)\big){\rm .}$$ 
Pour les choix de $a=\frac{1,742}{3}$, qui est compatible avec les conditions $a\leqslant 1$ et $3a\geqslant 1$,  et $\delta=0,0069$, nous obtenons
\begin{equation}
\label{zone 5}
\sum\limits_{\bfx\in \mathcal{D}_q\cap \mathcal{R}(\xi)}{\sum\limits_{\substack{d_1d_2\mid F(\bfx)\\ (d_1,d_2) \in 5}}{\chi(d_1)\chi^2(d_2)}} \ll ||F||^{\varepsilon}(\sigma^2+\vartheta^2)\xi^2\mathcal{L}(\xi)^{\alpha/2}(\log_2 \xi)^{O(1)} (\log \xi)^{-\eta}
\end{equation}
où $\eta$ est défini en \eqref{eta}. \\
\goodbreak

La proposition \ref{prop} s'obtient en reportant les estimations \eqref{zone 1}, \eqref{zone 8}, \eqref{zone 2}, \eqref{zone 4}, \eqref{zone 3+6}, \eqref{zone 7} et \eqref{zone 5} dans la formule \eqref{42} avec le choix $\delta=0,0069$. \\

Nous utilisons la Proposition \ref{prop} pour estimer les termes $Q_2(F_{\beta,d_1,d_2},\xi,\mathcal{R}_{\beta,d_1,d_2})$ dans la formule \eqref{eq Q1} lorsque $d_1$ et $d_2$ sont majorés par une puissance de $\log \xi$. La contribution des couples $(d_1,d_2)$ complémentaires  est traitée par le Lemme~\ref{conv} et le Lemme \ref{conv1}. Par ailleurs, en reprenant les notations de la Proposition \ref{prop}, nous utilisons l'égalité suivante
$$\varrho_F^+(p)=\varrho_J^+(p){\rm ,}$$
valable pour tout $p$  tel que $(p,D\mbox{d\'{e}t}(M))=1$. Cela fournit
\begin{align}
\label{57bis}
\begin{split}
Q_1(F,\xi,\mathcal{R})=&\frac{1}{q}{\rm vol}(\mathcal{R})\xi^2L(1,\chi)L(1,\chi^2)K_1(F)\sum\limits_{d_1 \mid q^{\infty}}{\sum\limits_{\alpha \in G_1}{\sum\limits_{d_2 \mid q^{\infty}}\frac{|W_{\alpha,d_1,d_2}|}{d_1d_2}}}\\
&+O\Big(\frac{||F||^{\varepsilon}(\sigma^2+\vartheta^2)\xi^2}{(\log \xi)^{\eta}}\Big){\rm .}
\end{split}
\end{align}
En reportant cette estimation dans l'égalité \eqref{eq Q'}, nous obtenons finalement
\begin{align*}
Q(F,\xi,\mathcal{R})=&\frac{1}{q}{\rm vol}(\mathcal{R})\xi^2L(1,\chi)L(1,\chi^2)K_1(F)\sum\limits_{d\mid q^{\infty}}{\frac{1}{d^2}\sum\limits_{d_1 \mid q^{\infty}}{\sum\limits_{\alpha \in G_1}{\sum\limits_{d_2 \mid q^{\infty}}\frac{|W_{\alpha,d_1,d_2}|}{d_1d_2}}}}\\
&+O\Big(\frac{||F||^{\varepsilon}(\sigma^2+\vartheta^2)\xi^2}{(\log \xi)^{\eta}}\Big)\\
\end{align*}
Cela achève la démonstration du Théorème \ref{theo 1} et nous permet de déterminer $K(F)$.
\begin{align}
\begin{split}
\label{K(T)}
K(F)&=\frac{3\varphi(q)}{q^2}\prod\limits_{p\nmid q}{\Big(\sum\limits_{\nu\geqslant 0}{\frac{\varrho_F^+(p^{\nu})}{p^{2\nu}}(\chi*\chi^2)(p^{\nu})}\Big)}\sum\limits_{d\mid q^{\infty}}{\frac{1}{d^2}\sum\limits_{d_1 \mid q^{\infty}}{\sum\limits_{\alpha \in G_1}{\sum\limits_{d_2 \mid q^{\infty}}\frac{|W_{\alpha,d_1,d_2}|}{d_1d_2}}}}\\
&=\prod\limits_{p\nmid q}{K_p(F)}K_q(F){\rm ,}
\end{split}
\end{align}
avec
\begin{equation}
K_q(F):=\frac{3\varphi(q)}{q^2}\sum\limits_{d\mid q^{\infty}}{\frac{1}{d^2}\sum\limits_{d_1 \mid q^{\infty}}{\sum\limits_{\alpha \in G_1}{\sum\limits_{d_2 \mid q^{\infty}}\frac{|W_{\alpha,d_1,d_2}|}{d_1d_2}}}}{\rm .}
\end{equation}
\end{proof}
Afin de démontrer l'égalité \eqref{K_q} portant sur la constante $K_q(F)$, nous rappelons la définition des ensembles suivants.
$$\mathcal{E}:=\bigcup\limits_{\alpha \in G_1}\{n \in \mathbb{N}^*\mbox{ : } \exists d\mid q^{\infty}\mbox{, } n\equiv \alpha d \bmod dq\}$$
où $G_1={\rm Ker}(\chi)$ et pour $d \mid q^{\infty}$, nous notons 
$\mathcal{E}_d$ la projection de $\mathcal{E}$ sur $\mathbb{Z}/d\mathbb{Z}$, c'est-à-dire
$$\mathcal{E}_d=\bigcup\limits_{\alpha \in G_1}\{n \in \mathbb{Z}/d\mathbb{Z}\mbox{ : } \exists d_1\mid q^{\infty}\mbox{, } n\equiv \alpha d_1 \bmod (d_1q,d)\}{\rm .}$$
Nous pouvons alors énoncer le lemme suivant.
\begin{lemme}
\label{lemme K_q}
Nous avons 
$$K_q(F)=\lim\limits_{k\rightarrow \infty}\frac{3}{q^{2k}}\Big|\Big\{\bfx \in (\mathbb{Z}/q^k\mathbb{Z})^2\mbox{ : } F(\bfx) \in \mathcal{E}_{q^k}\Big\}\Big|{\rm .}$$
\end{lemme}
\begin{proof}
Nous écrivons 
\begin{align*}
&\frac{1}{q^{2k}}\Big|\Big\{\bfx \in (\mathbb{Z}/q^k\mathbb{Z})^2\mbox{ : } F(\bfx) \in \mathcal{E}_{q^k}\Big\}\Big|\\
=& \frac{1}{q^{2k}}\sum\limits_{d\mid q^k}{\Big|\Big\{\bfx \in (\mathbb{Z}/(q^k/d)\mathbb{Z})^2\mbox{ : } d^3F(\bfx) \in \mathcal{E}_{q^k} \mbox{, }(\bfx,q^k/d)=1\Big\}\Big|}{\rm .}
\end{align*}
Lorsque $d \mid q^{[k/4]}$, la condition $(\bfx ,q^k/d)=1$ est équivalente à $(\bfx,q)=1$. Le cas complémentaire peut être aisément vu comme un terme d'erreur. Ainsi 
\begin{align*}
&\frac{1}{q^{2k}}\Big|\Big\{\bfx \in (\mathbb{Z}/q^k\mathbb{Z})^2\mbox{ : } F(\bfx) \in \mathcal{E}_{q^k}\Big\}\Big|\\
=& \frac{1}{q^{2k}}\sum\limits_{d\mid q^{[k/4]}}{\Big|\Big\{\bfx \in (\mathbb{Z}/(q^k/d)\mathbb{Z})^2\mbox{ : } d^3F(\bfx) \in \mathcal{E}_{q^k} \mbox{, }(\bfx,q)=1\Big\}\Big|}+o(1)\\
= &\frac{1}{q^{2k}}\sum\limits_{d\mid q^{[k/4]}}{\sum\limits_{d_1 \mid q^k/d}|A_{d,d_1,k}|}+o(1)
\end{align*}
où
$$A_{d,d_1,k}:=\Big\{\bfx \in (\mathbb{Z}/(q^k/d)\mathbb{Z})^2\mbox{ : } d^3F(\bfx) \in \mathcal{E}_{q^k} \mbox{, }m=d_1m' \mbox{, }(m',q^k/dd_2)=1 \mbox{, }(d_1m',n,q)=1\Big\}{\rm .}$$
Lorsque $d_2 \mid q^{k-1}/d$, la condition $(m' ,q^k/dd_2)=1$ est équivalente à $(m',q)=~1$, nous pouvons en déduire que la condition $(d_1m',n,q)=1$ est équivalente à la condition $(n,d_1)=1$. Le cas complémentaire peut être aisément vu comme un terme d'erreur. Ainsi 
\begin{align*}
&\frac{1}{q^{2k}}\Big|\Big\{\bfx \in (\mathbb{Z}/q^k\mathbb{Z})^2\mbox{ : } F(\bfx) \in \mathcal{E}_{q^k}\Big\}\Big|\\
= &\frac{1}{q^{2k}}\sum\limits_{d\mid q^{[k/4]}}{\sum\limits_{d_1 \mid q^{k-1}/d}{|A_{d,d_1,k}|}}+o(1)
\end{align*}
où nous avons réécrit
$$A_{d,d_1,k}:=\Big\{\bfx \in (\mathbb{Z}/(q^k/d)\mathbb{Z})^2\mbox{ : } d^3F(\bfx) \in \mathcal{E}_{q^k} \mbox{, }m=d_1m' \mbox{, }(m',q)=1 \mbox{, }(d_1,n)=1\Big\}{\rm .}$$
Lorsque $d\mid q^{[k/4]}$, la condition $d^3F(\mathbf{x}) \in \mathcal{E}_{q^k}$ est équivalente à l'existence d'un entier $d_2\mid q^{\infty}$ et d'un $\alpha \in G_1$ tel que 
$$F(\mathbf{x})\equiv \alpha d_2 \bmod (d_2q,q^k/d^3){\rm .}$$
La contribution des entiers $d_2$ tels que $d_2\nmid q^{k-1}/d^3$ étant négligeable, nous pouvons écrire
 \begin{align*}
&\frac{1}{q^{2k}}\Big|\Big\{\bfx \in (\mathbb{Z}/q^k\mathbb{Z})^2\mbox{ : } F(\bfx) \in \mathcal{E}_{q^k}\Big\}\Big|\\
= &\frac{1}{q^{2k}}\sum\limits_{d\mid q^{[k/4]}}{\sum\limits_{d_1 \mid q^{k-1}/d}{\sum\limits_{d_2 \mid q^{k-1}/d^3}{|B_{d,d_1,d_2,k}|}}}+o(1)
\end{align*}
où 
$$B_{d,d_1,d_2,k}:=\Big\{\bfx \in (\mathbb{Z}/(q^k/d)\mathbb{Z})^2\mbox{ : } F(\bfx) \in E_{d_2} \mbox{, }m=d_1m' \mbox{, }(m',q)=1 \mbox{, }(n,d_1)=1\Big\}$$
avec
$$E_{d_2}:=\bigcup\limits_{\alpha \in G_1}\{n \in \mathbb{N}^*\mbox{ : } \mbox{, } n\equiv \alpha d_2 \bmod d_2q\}{\rm .}$$
Il existe $\varphi(q^k/dd_1)$ entiers $m$ dans  $\mathbb{Z}/(q^k/d)\mathbb{Z}$ s'écrivant $d_1m'$ avec $(m',q)=1$, et il existe $q^{k-1}/dd_1\sum\limits_{\alpha \in G_1}{|W_{\alpha,d_1,d_2}|}$ éléments $\beta$ de $\mathbb{Z}/(q^k/d)\mathbb{Z}$ tels que 
$$F(d_1,\beta) \equiv \alpha d_2 \bmod d_2 q$$
 pour un certain $\alpha \in G_1$avec $(d_1,\beta)=1$. La donnée d'un tel couple $(m',\beta)$ détermine de manière unique chaque élément de l'ensemble $B_{d,d_1,d_2,k}$ en posant $\bfx=(d_1m',m'\beta)\in (\mathbb{Z}/(q^k/d)\mathbb{Z})^2$. \\
 
 Cela fournit l'estimation suivante
\begin{align*}
&\frac{1}{q^{2k}}\Big|\Big\{\bfx \in (\mathbb{Z}/q^k\mathbb{Z})^2\mbox{ : } F(\bfx) \in \mathcal{E}_{q^k}\Big\}\Big|\\
=&\frac{1}{q^{2k}}\sum\limits_{d\mid q^{[k/4]}}{\sum\limits_{d_1 \mid q^{k-1}/d^3}{\sum\limits_{d_2 \mid q^{k-1}/d}{\sum\limits_{\alpha \in G_1}{\varphi\Big(\frac{q^k}{dd_2}\Big)\frac{q^{k-1}}{dd_1}|W_{\alpha,d_1,d_2}|}}}}+o(1)\\
= &\frac{\phi(q)}{q^2}\sum\limits_{d\mid q^{[k/4]}}{\sum\limits_{d_1 \mid q^{k-1}/d^3}{\sum\limits_{d_2 \mid q^{k-1}/d}{\sum\limits_{\alpha \in G}{\frac{|W_{\alpha,d_1,d_2}|}{d^2d_1d_2}}}}}+o(1){\rm .}\\
\end{align*}
Par passage à la limite, nous obtenons 
$$K_q(F)=\lim\limits_{k\rightarrow \infty}\frac{3}{q^{2k}}\Big|\Big\{\bfx \in (\mathbb{Z}/q^k\mathbb{Z})^2\mbox{ : } F(\bfx) \in \mathcal{E}_{q^k}\Big\}\Big|{\rm ,}$$
ce qui démontre l'égalité \eqref{K_q}.
\end{proof}

\section{Interprétation géométrique de $K(F)$}
Nous supposons dans toute cette partie que l'anneau des entiers $O_{\mathbb{K}}$ est principal. Notre objectif ici est de démontrer la dernière assertion du Théorème \ref{theo 1}.
\subsection{Interprétation géométrique de $K_q(F)$}

Nous notons $G_1={\rm Ker}(\chi)$ qui est un sous groupe de $(\mathbb{Z}/q\mathbb{Z})^{\times}$ d'indice $3$.\\

Nous utilisons l'égalité du Lemme \ref{lemme K_q} portant sur $K_q(F)$ afin de décomposer cette quantité en produit de facteurs non archimédiens. \\

Pour cela, nous devons fixer une $\mathbb{Z}$-base de
$O_{\mathbb{K}}$, $(\omega_1,\omega_2,\omega_3)$. Cela nous permet de définir un polynôme homogène à $3$ variables $P$ en posant
$$P(y,z,t)=N_{\mathbb{K}/\mathbb{Q}}(y\omega_1+z\omega_2+t\omega_3){\rm .}$$
\begin{prop}
\label{prop 1}
Soient $n\geqslant 1$ et $A\geqslant 1$ deux entiers tels que pour tout $p\mid q$, nous avons
$n\geqslant 1+v_p(A)$ et tels que la partie de $A$ qui est première avec $q$ appartienne à $G_1$. Si $O_{\mathbb{K}}$ est principal, nous avons
$$\Big|\Big\{(y,z,t) \in (\mathbb{Z}/q^n\mathbb{Z})^3\mbox{ : } P(y,z,t)=A\Big\}\Big|=3q^{2n}{\rm .}$$
\end{prop}

Pour démontrer cette proposition, nous aurons besoin du lemme suivant.

\begin{lemme}
Soit $n$ un entier premier à $q$. Les assertions suivantes sont équivalentes.
\begin{itemize}
\item $\chi(n)=1$,
\item Il existe $(y,z,t) \in \mathbb{Z}^3$ tels que 
$$n\equiv P(y,z,t) \bmod q{\rm .}$$
\end{itemize}
\end{lemme}

\begin{proof}
Ce lemme se démontre facilement en considérant un nombre premier congru à $n$ modulo $q$ et en notant que la fonction $1*\chi*\chi^2(n)$ compte le nombre d'idéaux de norme $n$. L'anneau $O_{\mathbb{K}}$ étant principal, cela fournit le résultat.
\end{proof}

La seconde assertion est une propriété locale, il suffit de la vérifier pour chaque nombre premier $p$ divisant $q$. Ce lemme nous permet donc d'affirmer que $G_1$ s'écrit de la façon suivante.
$$G_1=\prod\limits_{p\mid q}{G_p}{\rm ,}$$
où $G_p$ est un sous groupe de $(\mathbb{Z}/p^{v_p(q)}\mathbb{Z})^{\times}$. Nous notons $r_p$ son indice. Nous avons dans ce cas
$$\prod\limits_{p\mid q}{r_p}=3{\rm .}$$\\

La proposition \ref{prop 1} peut donc être déduite de la proposition suivante.

\begin{prop}
Soit $p$ un nombre premier divisant $q$. Soient $n\geqslant 1$ et $A\geqslant 1$ tels que
$n\geqslant v_p(q)+v_p(A)$ et tels que la partie de $A$ qui est première avec $p$ appartienne à $G_p$. Sous l'hypothèse que $O_{\mathbb{K}}$ est principal, nous avons
$$\Big|\Big\{(y,z,t) \in (\mathbb{Z}/p^n\mathbb{Z})^3\mbox{ : } P(y,z,t)=A\Big\}\Big|=r_pp^{2n}{\rm .}$$
\end{prop}
\begin{proof}
Notons $E(p,n)$ l'ensemble 
$$E(p,n)=\{(y,z,t) \in (\mathbb{Z}/p^n\mathbb{Z})^3\mbox{ : } p\nmid P(y,z,t) \}{\rm .}$$

La condition $p\nmid P(y,z,t)$ ne dépend que de la congruence de $y,z$ et $t$ modulo $p$. Par ailleurs, comme $(1*\chi*\chi^2)(p)=1$, il n'existe qu'un idéal $\mathfrak{B}_p$ de norme $p$. Cet idéal est engendré par un élément $\beta_p$. La condition $p\nmid P(y,z,t)$ est donc équivalente à la condition $\beta_p \nmid y\omega_1+z\omega_2+t\omega_3$ dans $O_{\mathbb{K}}$. De plus, le choix d'une base nous permet d'expliciter l'isomorphisme.

\begin{align*}
\psi &:&& (\mathbb{Z}/p\mathbb{Z})^3 && \longrightarrow &&O_{\mathbb{K}}/pO_{\mathbb{K}} \\
&&& (y,z,t) && \mapsto && y\omega_1+z\omega_2+t\omega_3
\end{align*}
En notant $\pi$ la surjection canonique
$$\pi: O_{\mathbb{K}}/pO_{\mathbb{K}}\longrightarrow O_{\mathbb{K}}/\mathfrak{B}_p{\rm ,}$$ 
l'application $\pi\circ \psi$ est un morphisme de groupes surjectif et puisque 
$$O_{\mathbb{K}}/\mathfrak{B}_p \approx \mathbb{Z}/p\mathbb{Z}{\rm ,}$$ 
nous avons
$$|E(p,1)|=p^3-p^2{\rm .}$$
donc 
$$|E(p,n)|=p^{3n-1}(p-1){\rm .}$$

Il est par ailleurs possible de munir $E(p,n)$ d'une structure de groupe multiplicatif via le produit usuel de $O_{\mathbb{K}}$. Cette structure de groupe fait de l'application suivante
\begin{align*}
\phi &:&& E(p,n) && \mapsto &&(\mathbb{Z}/p^n\mathbb{Z})^{\times} \\
&&& (y,z,t) && \mapsto && P(y,z,t)
\end{align*}
un morphisme de groupes dont l'image est l'ensemble 
des éléments dont leur réduction modulo $p^{v_p(q)}$ appartient à $G_p$. Cet ensemble est de cardinal $p^{n-v_p(q)}p^{v_p(q)-1}(p-1)/r_p$. Le noyau de ce morphisme est donc de cardinal $r_pp^{2n}$, ce qui démontre le résultat dans le cas où $A$ est premier avec $p$.

Nous déduisons le cas général en raisonnant par récurrence sur la valuation p-adique de $A$, au moyen du diagramme commutatif suivant et des isomorphismes entre les groupes $ (\mathbb{Z}/p^{k}\mathbb{Z})^3$ et $O_{\mathbb{K}}/p^{k}O_{\mathbb{K}}$.

$$\xymatrix{
    O_{\mathbb{K}}/\mathfrak{B}_p^{3n-1}\ar[r]^{\pi_n} \ar[d]_{.\beta_p} &O_{\mathbb{K}}/p^{n-1}O_{\mathbb{K}} \ar[r]^{N_{\mathbb{K}/\mathbb{Q}}}  & (\mathbb{Z}/p^{n-1}\mathbb{Z}) \ar[d]^{.p} \\
    O_{\mathbb{K}}/p^{n}O_{\mathbb{K}}  \ar[rr]_{N_{\mathbb{K}/\mathbb{Q}}}&& (\mathbb{Z}/p^n\mathbb{Z})
  }
  $$
  où $.\beta_p$ désigne la multiplication  par $\beta_p$ dans $O_{\mathbb{K}}$, tandis que $\pi_n$ désigne la surjection canonique de $O_{\mathbb{K}}/\mathfrak{B}_p^{3n-1}$ vers $O_{\mathbb{K}}/p^{n-1}O_{\mathbb{K}}$. Les applications $.\beta_p$ et $.p$ sont injectives tandis que le noyau de $\pi_n$ est de cardinal $p^2$. Enfin, l'hypothèse $n\geqslant v_p(A)+v_p(q)$ montre que si $v_p(A)\geqslant 1$, alors l'unique antécédent de $A$ par la multiplication par $p$ est également bien défini modulo $p^{v_p(q)}$, ce qui nous permet d'appliquer la récurrence.
   \end{proof}
La proposition \ref{prop 1} permet alors d'écrire 
  \begin{align*}
   K_q(F)&=\lim\limits_{k\rightarrow \infty}&&\frac{3}{q^{2k}}\Big|\Big\{\bfx \in (\mathbb{Z}/q^k\mathbb{Z})^2\mbox{ : } F(\bfx) \in \mathcal{E}_{q^k}\Big\}\Big|\\
   &=\lim\limits_{k\rightarrow \infty}&&\frac{3}{q^{2k}}\Big(\sum\limits_{\substack{\bfx \in (\mathbb{Z}/q^k\mathbb{Z})^2\\ F(\bfx)\in \mathcal{E}_{q^k} \\ \forall p\mid q\;\; v_p(F(\bfx))\leqslant k-1}}{1}+\sum\limits_{\substack{\bfx \in (\mathbb{Z}/q^k\mathbb{Z})^2\\ F(\bfx)\in \mathcal{E}_{q^k} \\ \exists p\mid q\;\; v_p(F(\bfx))> k-1}}{1}\Big)\\
   &=\lim\limits_{k\rightarrow \infty}&&\frac{3}{q^{2k}}\sum\limits_{\substack{\bfx \in (\mathbb{Z}/q^k\mathbb{Z})^2\\ F(\bfx)\in \mathcal{E}_{q^k} \\ \forall p\mid q\;\; v_p(F(\bfx))\leqslant k-1}}{\!\!\!\frac{|\{(y,z,t)\in (\mathbb{Z}/q^k\mathbb{Z})^3\mbox{ : }P(y,z,t)\equiv F(\bfx)\bmod q^k\}|}{3q^{2k}}}\\
   &&&+O\Big(\sum\limits_{p\mid q}{\frac{\varrho_F^+(p^{k})}{p^{2k}}}\Big)\\
   &=\lim\limits_{k\rightarrow \infty}&&\frac{1}{q^{4k}}\Big|\Big\{(\bfx,y,z,t) \in (\mathbb{Z}/q^k\mathbb{Z})^5\mbox{ : } F(\bfx)\equiv P(y,z,t) \bmod q^k\Big\}\Big|{\rm .}
   \end{align*}
 Ainsi,
  $$K_q(F)=\prod\limits_{p\mid q}{K_p(F)}$$
  où
  $$K_p(F)=\lim\limits_{k\rightarrow \infty}\frac{1}{p^{4k}}\Big|\Big\{(\bfx,y,z,t) \in (\mathbb{Z}/p^k\mathbb{Z})^5\mbox{ : } F(\bfx)\equiv P(y,z,t)\bmod p^k\Big\}\Big|{\rm .}$$
  
\subsection{Interprétation géométrique de $K_p(F)$ lorsque $p$ ne divise pas $q$}

Rappelons la définition de $K_p(F)$ lorsque $p$ ne divise pas $q$.
\begin{definition}
Nous avons, pour tout $p\nmid q$
\begin{equation}
K_p(F):=\Big(1-\frac{\chi(p)}{p}\Big)\Big(1-\frac{\chi^2(p)}{p}\Big)\sum\limits_{k\geqslant 0}{\frac{\varrho_F^+(p^k)}{p^{2k}}(\chi*\chi^2)(p^k)}
\end{equation}
où
$$\varrho_F^+(p^k):=\Big|\Big\{\bfx \in (\mathbb{Z}/p^k\mathbb{Z})^2\mbox{ : } F(\bfx)\equiv 0\bmod p^k\Big\}\Big|{\rm .}$$
\end{definition}
\begin{definition}
Pour tout $p$ premier, nous notons
$$K_{pg}(F):=\lim\limits_{k\rightarrow \infty}\frac{1}{p^{4k}}\Big|\Big\{(\bfx,y,z,t) \in (\mathbb{Z}/p^k\mathbb{Z})^5\mbox{ : } F(\bfx)\equiv P(y,z,t)\bmod p^k\Big\}\Big|{\rm .}$$
\end{definition}
Notre objectif dans cette section est de démontrer la proposition suivante
\begin{prop}
Pour tout $p\nmid q$, nous avons
\begin{equation}
\label{K_pg}
K_p(F)= K_{pg}(F) {\rm .}
\end{equation}
\end{prop}
Notons également pour la suite, pour $A$ et $k$ deux entiers, et $p$ un nombre premier
$$S(A,p^k):=\Big|\Big\{(y,z,t) \in (\mathbb{Z}/p^k\mathbb{Z})^3\mbox{ : } P(y,z,t)\equiv A \bmod p^k\Big\}\Big|$$
où $P$ est défini dans la section précédente à partir d'une $\mathbb{Z}$-base de $O_{\mathbb{K}}$. La démonstration de la proposition ci dessus passe par un calcul explicite de $S(A,p^k)$, analogue aux calculs menés dans \cite{BB1}, \cite{BB2} et \cite{HB}.
\subsubsection{Cas $\chi(p)\neq 1$}
Dans ce cas, $pO_{\mathbb{K}}$ est un idéal premier de $O_{\mathbb{K}}$, nous pouvons alors énoncer le lemme suivant.
\begin{lemme}
Soient $A$ et $k$ deux entiers. Si $\chi(p)\neq 1$, nous avons
$$
S(A,p^k) = \left\{
    \begin{array}{ll}
        p^{2k}\Big(1+\frac{1}{p}+\frac{1}{p^2}\Big) & \mbox{si } 3\mid v_p(A) \mbox{ et } v_p(A)<k \\
        0 & \mbox{si } 3\nmid v_p(A) \mbox{ et }  v_p(A)<k \\
        p^{3[2k/3]}& \mbox{si } v_p(A)\geqslant k
    \end{array}
\right.
$$
\end{lemme}

\begin{proof}
Nous utilisons de nouveau l'isomorphisme
$$ \psi : (\mathbb{Z}/p^k\mathbb{Z})^3\longrightarrow   O_{\mathbb{K}}/p^kO_{\mathbb{K}}{\rm .}$$
Si $v_p(A)=0$ et
$$N_{\mathbb{K}/\mathbb{Q}}(y\omega_1+z\omega_2+t\omega_3)\equiv A \bmod p^k$$
alors $y\omega_1+z\omega_2+t\omega_3 \in (O_{\mathbb{K}}/p^kO_{\mathbb{K}})^{\times}$. L'application $N_{\mathbb{K}/\mathbb{Q}}$ induit un morphisme surjectif de groupes multiplicatifs
$$N_{\mathbb{K}/\mathbb{Q}}:(O_{\mathbb{K}}/p^kO_{\mathbb{K}})^{\times}\longrightarrow (\mathbb{Z}/p^k\mathbb{Z})^{\times}$$
ce qui nous permet d'écrire
$$S(A,p^k)=p^{2k}\Big(1+\frac{1}{p}+\frac{1}{p^2}\Big)$$
si $v_p(A)=0$. Nous pouvons raisonner par récurrence sur $v_p(A)$ pour conclure dans le cas où $v_p(A)<k$. En fait, si $v_p(A)\geqslant 3$, et 
$$N_{\mathbb{K}/\mathbb{Q}}(y\omega_1+z\omega_2+t\omega_3)\equiv A \bmod p^k$$
alors $p$ divise $y$, $z$ et $t$. Ainsi
\begin{align*}
S(A,p^k)&=\Big|\Big\{(y',z',t') \in (\mathbb{Z}/p^{k-1}\mathbb{Z})^3\mbox{ : } P(y',z',t')\equiv \frac{A}{p^3} \bmod p^{k-3}\Big\}\Big|\\
&=p^6S(\frac{A}{p^3},p^{k-3})
\end{align*}
Le cas $v_p(A)\geqslant k$ se démontre par récurrence sur $k$ au moyen de l'égalité ci-dessus et des conditions initiales
$$S(0,1)=S(0,p)=1{\rm ,} \;\;\;\;\; S(0,p^2)=p^3{\rm .}$$
\end{proof}
Nous pouvons donc écrire 
\begin{align*}
K_{pg}(F)&=\lim\limits_{\nu\rightarrow \infty}\frac{1}{p^{2\nu}}\sum\limits_{k=0}^{\nu-1}{\sum\limits_{\substack{\bfx \in (\mathbb{Z}/p^k\mathbb{Z})^2 \\ v_p(F(\bfx))=k}}{\frac{1}{p^{2\nu}}S(F(\bfx),p^{\nu})}}+o(1) \\
&=\Big(1+\frac{1}{p}+\frac{1}{p^2}\Big)\lim\limits_{\nu\rightarrow \infty}\frac{1}{p^{2\nu}}\sum\limits_{\substack{k=0 \\ 3\mid k}}^{\nu-1}{\big(\varrho_F^+(p^k)p^{2(\nu-k)}-\varrho_F^+(p^{k+1})p^{2(\nu-k-1)}\big)}\\
&= \Big(1+\frac{1}{p}+\frac{1}{p^2}\Big)\sum\limits_{\substack{k\geqslant 0 \\ 3\mid k}}{\Big(\frac{\varrho_F^+(p^k)}{p^{2k}}-\frac{\varrho_F^+(p^{k+1})}{p^{2(k+1)}}\Big)}{\rm .}
\end{align*}

Par ailleurs, puisque $\chi(p)\neq 1$ et que $\chi$ est d'ordre $3$, il est facile de vérifier que
$$\Big(1-\frac{\chi(p)}{p}\Big)\Big(1-\frac{\chi^2(p)}{p}\Big)=1+\frac{1}{p}+\frac{1}{p^2}$$
et
$$
(\chi*\chi^2)(p^k) = \left\{
    \begin{array}{ll}
        1 & \mbox{si } k\equiv 0 \bmod 3 \\
        -1 & \mbox{si } k\equiv 1 \bmod 3 \\
        0& \mbox{si } k\equiv 2 \bmod 3
    \end{array}
\right.
$$
Ainsi, nous avons

$$K_{pg}(F)=\Big(1-\frac{\chi(p)}{p}\Big)\Big(1-\frac{\chi^2(p)}{p}\Big)\sum\limits_{k\geqslant 0}{\frac{\varrho_F^+(p^k)}{p^{2k}}(\chi*\chi^2)(p^k)}=K_p(F){\rm .}$$

\subsubsection{Cas $\chi(p)=1$}
Dans ce cas, $pO_{\mathbb{K}}$ se décompose en produit de trois idéaux premiers 
$$pO_{\mathbb{K}}=\mathfrak{B}_1\mathfrak{B}_2\mathfrak{B}_3{\rm .}$$. Nous pouvons alors énoncer le lemme suivant.
\begin{lemme}
Soient $A$ et $k$ deux entiers. Nous avons
$$
S(A,p^k) = \left\{
    \begin{array}{ll}
        p^{2k}\Big(1-\frac{1}{p}\Big)^2\binom{v_p(A)+2}{2} & \mbox{si }  v_p(A)<k \\
        p^{2k}\Big\{\frac{k(k+1)}{2}\Big(1-\frac{1}{p}\Big)^2+k\Big(1-\frac{1}{p}\Big)+1\Big\} & \mbox{si }   v_p(A)\geqslant k 
    \end{array}
\right.
$$
\end{lemme}

\begin{proof}
Il s'agit de compter les éléments de norme $A \bmod p^k$ dans~$O_{\mathbb{K}}/p^kO_{\mathbb{K}}$.
Nous considérons l'isomorphisme d'anneaux
$$O_{\mathbb{K}}/p^kO_{\mathbb{K}}\approx O_{\mathbb{K}}/\mathfrak{B}_1^k \times O_{\mathbb{K}}/\mathfrak{B}_2^k \times O_{\mathbb{K}}/\mathfrak{B}_3^k{\rm .}$$
Si $v_p(A)=0$, l'application $N_{\mathbb{K}/\mathbb{Q}}$ induit de nouveau un morphisme de groupes surjectif
$$N_{\mathbb{K}/\mathbb{Q}}:(O_{\mathbb{K}}/p^kO_{\mathbb{K}})^{\times}\longrightarrow (\mathbb{Z}/p^k\mathbb{Z})^{\times}{\rm ,}$$
mais lorsque $\chi(p)=1$, le cardinal de $(O_{\mathbb{K}}/p^kO_{\mathbb{K}})^{\times}$ est $p^{3k}(1-\frac{1}{p})^3$. Ce qui démontre l'égalité
$$S(A,p^k)=p^{2k}(1-\frac{1}{p})^2$$
si $v_p(A)=0$. Dans le cas où $v_p(A)<k$, la condition
$$ N_{\mathbb{K}/\mathbb{Q}}(y\omega_1+z\omega_2+t\omega_3)\equiv A \bmod p^k$$
implique qu'il existe des exposants $(i,j,h)$ tels que
$$i+j+h=v_p(A)$$
 et un élément $\bfu$ premier avec $p$ tels que 
$$y\omega_1+z\omega_2+t\omega_3=\beta_1^i\beta_2^j\beta_3^h \bfu$$
avec $$\mathfrak{B}_{\ell}=(\beta_{\ell}){\rm ,}\;\;\;\; \ell=1,2,3{\rm .}$$
Un tel $\bfu$ est défini de manière unique dans
$$O_{\mathbb{K}}/\mathfrak{B}_1^{k-i} \times O_{\mathbb{K}}/\mathfrak{B}_2^{k-j} \times O_{\mathbb{K}}/\mathfrak{B}_3^{k-h}{\rm .}$$
Une fois fixés les exposants $(i,j,h)$, il est facile de voir que le nombre d'éléments $\bfu \in O_{\mathbb{K}}/\mathfrak{B}_1^{k-i} \times O_{\mathbb{K}}/\mathfrak{B}_2^{k-j} \times O_{\mathbb{K}}/\mathfrak{B}_3^{k-h}$ tels que 
$$N_{\mathbb{K}/\mathbb{Q}}(\bfu)\equiv \frac{A}{p^{v_p(A)}}\bmod p^{k-v_p(A)}$$
 est de cardinal $p^{2k}(1-\frac{1}{p})^2$. Il existe enfin $\binom{v_p(A)+2}{2}$ choix possibles pour les exposants $(i,j,h)$.
 Enfin, dans le cas où $v_p(A)\geqslant k$, l'ensemble des solutions correspond à un choix d'un triplet $(i,j,h)$ tels que $i+j+h\geqslant k$ et d'un élément $\bfu \in O_{\mathbb{K}}/\mathfrak{B}_1^{k-i} \times O_{\mathbb{K}}/\mathfrak{B}_2^{k-j} \times O_{\mathbb{K}}/\mathfrak{B}_3^{k-h}$ premier avec $p$.
\end{proof}

Nous pouvons donc écrire 
\begin{align*}
K_{pg}(F)&=\lim\limits_{\nu\rightarrow \infty}\frac{1}{p^{2\nu}}\sum\limits_{k=0}^{\nu-1}{\sum\limits_{\substack{\bfx \in (\mathbb{Z}/p^k\mathbb{Z})^2 \\ v_p(F(\bfx))=k}}{\frac{1}{p^{2\nu}}S(F(\bfx),p^{\nu})}}+o(1) \\
&=\Big(1-\frac{1}{p}\Big)^2\lim\limits_{\nu\rightarrow \infty}\frac{1}{p^{2\nu}}\sum\limits_{k=0}^{\nu-1}{\Big(\varrho_F^+(p^k)p^{2(\nu-k)}-\varrho_F^+(p^{k+1})p^{2(\nu-k-1)}\Big)\binom{k+2}{2}}\\
&= \Big(1-\frac{1}{p}\Big)^2\sum\limits_{k\geqslant 0 }{\frac{\varrho_F^+(p^k)}{p^{2k}}(k+1)}{\rm .}
\end{align*}

Par ailleurs, puisque $\chi(p)= 1$, il est facile de vérifier que
$$\Big(1-\frac{\chi(p)}{p}\Big)\Big(1-\frac{\chi^2(p)}{p}\Big)=\Big(1-\frac{1}{p}\Big)^2$$
et
$$
(\chi*\chi^2)(p^k) = k+1{\rm .}
$$
Ainsi, nous avons
$$K_{pg}(F)=\Big(1-\frac{\chi(p)}{p}\Big)\Big(1-\frac{\chi^2(p)}{p}\Big)\sum\limits_{k\geqslant 0}{\frac{\varrho_F^+(p^k)}{p^{2k}}(\chi*\chi^2)(p^k)}=K_p(F){\rm .}$$
\section*{Remerciements}

Je tiens particulièrement à remercier Marc Hindry, Kévin Destagnol, ainsi que mon directeur de thèse Régis de la Bretèche pour toutes les discussions que nous avons eues, ainsi que leurs conseils avisés qui m'ont permis d'écrire cet article.
\nocite{*}
\bibliographystyle{amsplain}


\bibliography{bibliofano}

\providecommand{\bysame}{\leavevmode\hbox to3em{\hrulefill}\thinspace}
\providecommand{\MR}{\relax\ifhmode\unskip\space\fi MR }
\providecommand{\MRhref}[2]{%
  \href{http://www.ams.org/mathscinet-getitem?mr=#1}{#2}
}
\providecommand{\href}[2]{#2}
\begin{thebibliography}{10}

\bibitem{BB2}
R.~de~la Bretèche and T.~D. Browning, \emph{Binary linear forms as sums of two
  squares}, Compositio Mathematica \textbf{144} (2008), no.~6, 1375--1402.

\bibitem{BB1}
\bysame, \emph{Binary forms as sums of two squares and ch\^atelet surfaces},
  2011.

\bibitem{B-T}
R.~de~la Bretèche and G.~Tenenbaum, \emph{Moyennes de fonctions arithmétiques
  de formes binaires}, Mathematika \textbf{58}, no.~2, 290--304.

\bibitem{B}
\bysame, \emph{Sur la conjecture de {M}anin pour certaines surfaces de
  {C}hâtelet}, Journal of the Institute of Mathematics of Jussieu \textbf{12}
  (2013), no.~4, 759--819.

\bibitem{C}
H.~Cohen, \emph{A course in computational algebraic numeber theory}, Graduate
  Texts in Mathematics, Springer, 1993.

\bibitem{D}
S.~Daniel, \emph{On the divisor-sum problem for binary forms}, Journal für die
  reine und angewandte Mathematik \textbf{1999} (1999), no.~507.

\bibitem{H-T}
R.R. Hall and G.~Tenenbaum, \emph{Divisors}, Cambridge Tracts in Mathematics,
  Cambridge University Press, 1988.

\bibitem{HB}
D.R. Heath-Brown, \emph{Linear relations amongst sums of two squares}, London
  Mathematical Society Lecture Note Series, pp.~133--176, Cambridge University
  Press, 2004.

\bibitem{H}
H.~Heilbronn, \emph{Zeta-functions and {L}-functions}, in Algebraic Number
  Theory, 204--230.

\bibitem{He}
K.~Henriot, \emph{Nair-tenenbaum bounds uniform with respect to the
  discriminant}, Mathematical Proceedings of the Cambridge Philosophical
  Society \textbf{152} (2012), no.~3, 405--424.

\bibitem{L}
A.~Lartaux, \emph{Sur la fonction {D}elta de {H}ooley associée à des
  caractères}, 2020, prépublication.

\end{thebibliography}
\addcontentsline{toc}{section}{Références}

\end{document}